\documentclass{aptpub}

\usepackage[colorlinks = true,linkcolor = blue,urlcolor  = blue,citecolor = blue,anchorcolor = blue]{hyperref}
\authornames{R.~AZA\"{I}S AND B. HENRY} 
\shorttitle{Estimation of spinal-structured trees} 

\usepackage{bbm}
\usepackage[english]{babel}
\usepackage[utf8]{inputenc}
\usepackage[T1]{fontenc}
\usepackage{dsfont}
\usepackage{latexsym}
\usepackage{mathrsfs}
\usepackage{marvosym}
\usepackage{graphicx}
\usepackage{color}
\usepackage{fullpage}
\usepackage{enumerate}
\usepackage{verbatim}
\usepackage{cprotect}
\usepackage{empheq}

\usepackage{enumitem}
\usepackage{titletoc}
\usepackage{graphics}
\usepackage{caption}
\usepackage{listings}
\usepackage{tikz}
\usetikzlibrary{decorations.text}
\usepackage{stmaryrd}
\usepackage{hyperref}
\usepackage{booktabs,array}

\makeatletter
\newcommand*\bigcdot{\mathpalette\bigcdot@{.5}}
\newcommand*\bigcdot@[2]{\mathbin{\vcenter{\hbox{\scalebox{#2}{$\m@th#1\bullet$}}}}}
\makeatother

\newcommand{\child}[1]{\mathcal{C}(#1)}
\newcommand{\pmax}{N}
\newcommand{\card}[1]{\##1}

\DeclareMathOperator*{\argmax}{arg\,max}

\newcommand{\ee}{\varepsilon}

\newcommand{\KL}[2]{d_{KL}\left(#1,#2\right)}
\newcommand{\dist}{\mu}
\newcommand{\spine}{\mathcal{S}}

\newcommand{\tree}{T}
\newcommand{\hatmu}{\widehat{\mu}_{h}}

\renewcommand{\S}[1]{S_{#1}}
\newcommand{\measure}{\mathcal{M}}
\newcommand{\bah}[2]{d_{B}\left(#1,#2\right)}
\newcommand{\spineset}[1]{\mathfrak{S}_{#1}}
\newcommand{\estS}[1]{\widehat{\mathcal{S}}_{#1}}
\newcommand{\B}[1]{\mathcal{B}#1}
\newcommand{\norm}[1]{\left\|#1\right\|}

\newcommand{\ball}[2]{B_{1}(#1,#2)}
\newcommand{\m}[1]{#1_{-}}
\newcommand{\mean}[1]{m(#1)}
\newcommand{\x}[1]{{\bf x}^{(#1)}}
\newcommand{\hatf}{\widehat{f}_{h}}
\newcommand{\hatspine}{\widehat{\mathcal{S}}_{h}}
\newcommand{\AH}[2]{\mathfrak{D}\left(#1,#2\right)}

\newcommand{\s}{\mathbf{s}}

\newcommand{\lag}{\mathbf{L}}
\newcommand{\lopt}{\Gamma}
\newcommand{\solopt}{\mathbf{S}}
\newcommand{\auxfonc}{K}
\newcommand{\feas}{\mathbf{F}}

\newcounter{compteleslignes}

\DeclareRobustCommand\paramalpha{\alpha} 
\DeclareRobustCommand\paramepsilon{\ee} 
\DeclareRobustCommand\numligne{} 
\newcounter{Optim}
\newenvironment{OptimEquation}
   {\stepcounter{Optim}%
     \addtocounter{equation}{-1}%
     \equation}
   {\endequation}

\begin{document}

\title{Maximum likelihood estimation\\for spinal-structured trees}

\authorone[Inria team MOSAIC]{Romain Aza\"{i}s}
\authortwo[IMT Nord Europe]{Beno\^it Henry} 

\addressone{\'ENS de Lyon, 46 All\'ee d'Italie, 69\,007 Lyon, France}
\emailone{romain.azais@inria.fr}

\addresstwo{Institut Mines-T\'el\'ecom, Univ.\ Lille, 59\,000 Lille, France}
\emailtwo{benoit.henry@imt-nord-europe.fr}

\begin{abstract}
	We investigate some aspects of the problem of the estimation of birth distributions (BD) in multi-type Galton-Watson trees (MGW) with unobserved types. More precisely, we consider two-type MGW called spinal-structured trees. This kind of tree is characterized by a spine of special individuals whose BD $\nu$ is different from the other individuals in the tree (called normal whose BD is denoted $\mu$). In this work, we show that even in such a very structured two-type population, our ability to distinguish the two types and estimate $\mu$ and $\nu$ is constrained by a trade-off between the growth-rate of the population and the similarity of $\mu$ and $\nu$. Indeed, if the growth-rate is too large, large deviations events are likely to be observed in the sampling of the normal individuals preventing us to distinguish them from special ones. 
	Roughly speaking, our approach succeeds if $r<\mathfrak{D}(\mu,\nu)$ where $r$ is the exponential growth-rate of the population and $\mathfrak{D}$ is a divergence measuring the dissimilarity between $\mu$ and $\nu$.
\end{abstract}



\newpage

\section{Introduction} \label{s:intro}

\subsection{Problem formulation}

A Galton-Watson tree with birth distribution $\mu$ is a random tree obtained recursively as follows: starting from the root, the number of children of any node is generated independently according to $\mu$. In light of \cite{bhat_adke_1981}, the log-likelihood of such a tree $T$ observed until generation $h$ is given by
$$\mathcal{L}^{\text{\tiny{GW}}}_h(\mu) = \sum_{v\in T,\,\mathcal{D}(v)<h} \log\,\mu(\#\mathcal{C}(v)),$$
where $\mathcal{D}(v)$ denotes the depth of node $v$, i.e. the length of the path from the root to $v$, and $\mathcal{C}(v)$ stands for the set of children of $v$. When the mean value $\mean{\mu}$ of the birth distribution $\mu$ is smaller than $1$, the model is said to be subcritical and the number of vertices of $T$ as well as its expectation are finite, i.e. the genealogy associated to $T$ becomes extinct. If a subcritical Galton-Watson model is conditioned to survive until (at least) generation $h$, the structure of the induced trees is changed according to Kesten's theorem \cite{abraham2014,kesten1986subdiffusive}. Indeed, the conditional distribution converges, when $h$ tends to infinity, towards the distribution of Kesten's tree defined as follows. Kesten's tree is a two-type Galton-Watson tree in which nodes can be normal or special such that:
\begin{itemize}
\item The birth distribution of normal nodes is $\mu$, while the one of special nodes $\nu$ is biased as,
\begin{equation}\label{eq:bias:kesten}
\forall\,k\geq0,~\nu(k) = \frac{k \mu(k)}{\mean{\mu}}.
\end{equation}
As for Galton-Watson trees, the numbers of children are generated independently.
\item All the children of a normal node are normal. Exactly one child of the children of a special node (picked at random) is special.
\end{itemize}
It should be noted that the set of children of a special node is non-empty since $\nu(0)=0$. Consequently, Kesten's tree consists in an infinite spine composed of special nodes, to which subcritical Galton-Watson trees of normal nodes (with birth distribution $\mu$) are attached. Following the reasoning presented in \cite{bhat_adke_1981} together with the form of the special birth distribution \eqref{eq:bias:kesten}, the log-likelihood of Kesten's tree is given by
\begin{eqnarray}
\mathcal{L}^{\text{\tiny{K}}}_h(\mu) &=& \sum_{v\notin\mathcal{S},\,\mathcal{D}(v)<h} \log\,\mu(\#\mathcal{C}(v)) + \sum_{v\in\mathcal{S},\,\mathcal{D}(v)<h} \log\,\nu(\#\mathcal{C}(v)) \nonumber \\
&=& \sum_{\mathcal{D}(v)<h} \log\,\mu(\#\mathcal{C}(v)) + \sum_{v\in\mathcal{S},\,\mathcal{D}(v)<h} \log\,\#\mathcal{C}(v) - h \log\,m(\mu) ,\nonumber
\end{eqnarray}
where $\mathcal{S}$ denotes the spine of $T$, i.e. the set of special nodes. Interestingly, maximizing the log-likelihood (with respect to $\mu$) does not require to observe the types of the nodes. Indeed, the term that involves the spine does not depend on the parameter of the model.

In this paper, we investigate spinal-structured trees which can be seen as a generalization of Kesten's tree. A spinal-structured tree is a two-type Galton-Watson tree, made of normal and special nodes, parameterized by a distribution $\mu$ and a non-trivial function $f:\mathbb{N}\to\mathbb{R}_+$, such that:
\begin{itemize}
\item The birth distribution of normal nodes is $\mu$, while the one of special nodes $\nu$ is biased as,
\begin{equation}\label{eq:biased:distrib}
\forall\,k\in\mathbb{N},~\nu(k) = \frac{f(k) \mu(k)}{\sum_{l\geq0} f(l) \mu(l)} ,
\end{equation}
assuming that the denominator is positive.
\item As for Kesten's tree, a normal node gives birth to normal nodes, while, if the set of children of a special node is non-empty, then exactly one of them (picked at random) is special.
\end{itemize}
Whenever $f(0)=0$, a spinal-structured tree admits an infinite spine made of special nodes, which gives its name to the model. It should be remarked that the model fails to be identifiable because the line spanned by $f$ defines the same probability measure $\nu$. As a consequence, without loss of generality, we assume
$$ \sum_{l\geq0} f(l)\mu(l) = 1.$$
Taking this into account, the log-likelihood of spinal-structured trees is given by
$$
\mathcal{L}_h^{\text{\tiny{SST}}} (\mu,f) = \sum_{\mathcal{D}(v)<h} \log\,\mu(\#\mathcal{C}(v)) + \sum_{v\in\mathcal{S},\,\mathcal{D}(v)<h} \log\,f(\#\mathcal{C}(v)).
$$

Whatever the birth distribution $\mu$, any biased distribution $\nu$ can be written as \eqref{eq:biased:distrib} with a wise choice of $f$ (except of course distributions $\nu$ such that, for some $k$, $\nu(k)>0$ and $\mu(k)=0$). The parameterization $(\mu,f)$ instead of $(\mu,\nu)$ makes clearer that spinal-structured trees form a generalization of Kesten's tree, which is obtained if and only if $f$ is linear, considering that $\mu$ is subcritical. In addition, Galton-Watson trees can be seen as spinal-structured trees assuming that $f$ is constant. Our goal in this work is to investigate the problem of estimating $\mu$ and $f$  through the maximization of $\mathcal{L}_h^{\text{\tiny{SST}}}$ without the knowledge of the types of the individuals. The main advantage of the parameterization $(\mu,f)$ is that it allows, as for Kesten's tree, to maximize the log-likelihood with respect to $\mu$ without observing the types of the nodes. However, maximizing it with respect to $f$ entails the observation of the types of the nodes.

\subsection{Motivation}

The motivation for this paper is twofold: first, it provides a step forward in the theoretical understanding of type identification in multi-type Galton Watson trees (MGW); second, it offers preliminary theoretical foundations for statistically testing whether population data have been conditioned to survive or not. These two points are detailed below.

	Spinal-structured trees can be seen as particular instances of two more general models. If the special individuals are interpreted as immigrants, the underlying population process is a Galton-Watson model with immigration given by $\nu$. And more generally, as every Galton-Watson process with immigration, it can also be seen as a particular case of MGW. The problem of the estimation of birth distributions in MGW has been heavily studied, e.g. \cite{MG1,MG2,MG4,MG3} and references therein, but in all these works the types of the individuals are assumed to be known. Very few works \cite{NM1,NM3,NM2} investigate this problem with unobserved types but none of these provide theoretical results: they only investigate numerical aspects. Using the special case of spinal-structured trees, this paper aims to demonstrate the theoretical difficulties involved in type estimation and propose a statistical strategy for dealing with them. In particular, we shall show that we are able to estimate the underlying parameters when population growth is not too large compared with the dissimilarity of the two birth distributions. This phenomenon is likely to hold true for more complicated problems.

When estimating the parameters of an observed population using a stochastic model, the latter must first be accurately chosen. To the best of our knowledge, even in the simple framework of Galton-Watson models, there is no method in the literature for rigorously determining from the data whether the population has been conditioned to survive or not. However, as aforementioned, estimating the parameters under the wrong model introduces significant biases that can lead to wrong conclusions about the population. Spinal-structured trees generalize both Galton-Watson trees ($f$ is constant, not investigated in the sequel of the paper) and Galton-Watson trees conditioned to survive ($f\propto\text{Id}$). By estimating $f$, and better by testing the shape of $f$, we can conclude which model to apply. The results of this paper will allow us to make progress in this direction (see in particular Subsection~\ref{ss:test}).

\subsection{Aim of the paper}

The present paper is dedicated to the development and the study of an estimation method for the unknown parameters $\mu$ and $f$, as well as the unknown type of the nodes, from the observation $T_h$ of one spinal-structured tree until generation $h$. The estimation algorithm that we derive in the sequel is based on the maximization of $\mathcal{L}_h^{\text{\tiny{SST}}}$ with the major difficulty that types are unobserved. Once the calculations are done, it can be succinctly described as follows.
\begin{enumerate}
\item Naive estimation of $\mu$:
$$
\widehat{\mu}_h(i) = \cfrac{1}{\card{T}_{h}}\sum_{v\in\tree_h}\mathbbm{1}_{\child{v}=i} .
$$
\item Estimation of the spine, i.e. the set of special nodes:
$$ \hatspine=\argmax_{\mathbf{s}\in\spineset{h}}\KL{\overline{\mathbf{s}}}{\B{\hatmu}},$$
where $\spineset{h}$ denotes the set of spine candidates (branches still alive at generation $h$), $\overline{\mathbf{s}}$ is the empirical distribution of the number of children along the spine candidate $\mathbf{s}$, $\B{\hatmu}(i)\propto i\,\hatmu(i)$, and $\KL{p}{q}$ denotes the Kullback-Leibler divergence between distributions $p$ and $q$.
\item Unbiased estimation of $\mu$ (without estimated special nodes $\hatspine$):
$$\hatmu^\star(i)=\cfrac{1}{\card{T}_{h}-h}\sum_{v\in \tree_h\setminus\hatspine}\mathbbm{1}_{\child{v}=i}.$$
\item Estimation of $f$:
$$\hatf(i)=\cfrac{1}{\hatmu^\star(i)h}\sum_{v\in\hatspine}\mathbbm{1}_{\child{v}=i}.$$
\end{enumerate}

Even in such a structured instance of MGW, the convergence of these estimates is far from easy to be established. We state in Theorem~\ref{thm:mainThm} that, if the distribution of surviving normal nodes is not too close to the special birth distribution $\nu$ compared to the exponential growth-rate of the tree, then $\widehat{\mu}^\star_h$ and $\widehat{f}_h$ almost surely converge towards $\mu$ and $f$. In addition, the recovered part of the spine is almost surely of order $h$ when $h$ goes to infinity. We insist on the fact that these two results are true whatever the growth-regime of the tree (subcritical $m(\mu)<1$, critical $m(\mu)=1$, or supercritical $m(\mu)>1$). Nevertheless, the reason behind these convergence results is not the same in the subcritical regime (where almost all the spine can be recovered in an algorithmic fashion) and in the critical and supercritical regimes (where the number of spine candidates explodes). The theoretical convergence properties related to the asymptotics of the estimators $\widehat{\mu}^\star_h$, $\widehat{f}_h$, and $\widehat{\mathcal{S}}_h$ are shown under the following main conditions, that are essential to the proofs of convergence:
\begin{itemize}
\item The maximum number of children in the tree is $N\geq1$, i.e. $\mu\in\measure$ where $\measure$ denotes the set of probability distributions on $\{0,\dots,N\}$. By construction \eqref{eq:biased:distrib}, $\nu$ also belongs to $\measure$.
\item $f(0)=0$ so that $\nu(0)=0$, i.e. the tree admits an infinite spine of special nodes.
\end{itemize}
For the sake of readability and conciseness of the proofs, we also assume:
\begin{itemize}
\item The support of $\mu$ is $\{0,\dots,N\}$.
\item $f(k)>0$ for any $k>0$, which implies that the support of $\nu$ is $\{1,\dots,N\}$.
\end{itemize}

The article is organized as follows. Section \ref{s:algo:spine} describes how some parts of the spine can be algorithmically recovered in a deterministic fashion. Section \ref{sec:MLS} is devoted to our estimation procedure and theoretical results:
\begin{itemize}
\item  Subsection \ref{ss:1stmu} for the preliminary estimation of $\mu$;
\item Subsection \ref{ss:ugly} for the identification of a candidate for the spine, named the Ugly Duckling;
\item Subsection \ref{ss:2stmu} for the final estimation of $\mu$ and the estimation of $f$;
\item Subsection \ref{ss:mainthm} for the statement of our main result: Theorem \ref{thm:mainThm}.
 \end{itemize}
The proof of Theorem \ref{thm:mainThm} in the subcritical case is done in Section \ref{s:subproof}. The proof in the supercritical case involves large deviation-type estimates for which we need information on the rate function. The rate function is studied in Section \ref{s:LDP} and the information needed stated in Theorem \ref{thm:optim}.
We finally consider the proof in the critical and supercritical cases in Section \ref{sec:mainproof}.
The final Section \ref{s:num} is devoted to numerical illustrations of the results (Subsection~\ref{ss:illus}) and an application to asymptotic tests for populations conditioned on surviving (Subsection~\ref{ss:test}). Appendix \ref{app:lem} concerns the proof of some intermediate lemmas.


\section{Algorithmic identification of the spine} \label{s:algo:spine}

Here we propose an algorithm to (at least partially) identify the spine of a spinal-structured tree $T$ observed until generation $h$. A node $v$ of $T$ is called observed if $\mathcal{D}(v)<h$. It means that the number of children of $v$ can be considered as part of the data available to reconstruct the spine of $T$ (even if the depth of these children is $h$). The tree restricted to the observed nodes is denoted $T_h$.

We will also need the notion of observed height of a subtree $T[v]$ of $T$. If $v$ is a node of $T$, $T[v]$ denotes the tree rooted at $v$ and composed of $v$ and all its descendants in $T$. In the literature on trees, the height $\mathcal{H}(T[v])$ of a subtree $T[v]$ is the length of the longest path from its root $v$ to its leaves. In the context of this work, one only observes $T$ until generation $h$ and thus the height of $T[v]$ can be unknown since the leaves of $T[v]$ can be unaccessible. For this reason, we define the observed height of $T[v]$ as
\begin{equation}\label{eq:def:H0}
	\mathcal{H}_o(T[v]) = \min( \mathcal{H}(T[v])~,~l-\mathcal{D}(v)),
\end{equation}
where $l$ is the length of the minimal path from $v$ to unobserved nodes. It should be remarked that $\mathcal{H}_o$ implicitly depends on $h$. Either $l=+\infty$ if $v$ has no unobserved descendant, or $l = h-\mathcal{D}(v)$. In addition, $v$ has no unobserved descendant if and only if $\mathcal{H}(T[v]) + \mathcal{D}(v) <h$.

The following result makes possible to partially identify the spine.
\begin{prop}\label{prop:id:spine}Let $T$ be a spinal-structured tree observed until generation $h$ and $v$ an observed node of $T$.
\begin{itemize}
\item If $\mathcal{H}_o(T[v])+\mathcal{D}(v) < h$, then $v$ is normal.
\item If $v$ is special, the children of $v$ are observed, and
$$\exists!\,c\in\mathcal{C}(v)~\text{such that}~\mathcal{H}_o(T[c])+\mathcal{D}(c)\geq h ,$$
then $c$ is special.
\end{itemize}
\end{prop}
\begin{proof}
The proof relies on the fact that special nodes have an infinite number of descendants. First, if $v$ is such that $\mathcal{H}_o(T[v])+\mathcal{D}(v) < h$, it means that its subtree has become extinct before generation $h$ and thus $v$ is normal. Second, if $v$ is special, exactly one of its children is special. All the subtrees rooted at the children of $v$ that become extinct are composed of normal nodes. Consequently, if only one subtree among its children has not become extinct before generation $h$, it is necessarily special.
\end{proof}

\noindent
It should be noticed that if an observed node is not covered by the two previous conditions, it can be either special or normal. Indeed, if the $c_k$'s are the children of a special node $v$ that do not become extinct before generation $h$, only one of them is special, while the other ones are normal. Actually, only the distribution of the subtrees rooted at the $c_k$'s can be used to differentiate them. An application of Proposition~\ref{prop:id:spine} is presented in Fig.\,\ref{fig:ex:spine}.

If a node $v$ at depth $\mathcal{D}(v) = h-1$ has been identified as special, i.e. if $v$ is the only node with children at depth $h-1$, it means that the spine has been algorithmically reconstructed, and is formed by $v$ and all its ascendants. Otherwise, if the type of two or more nodes at depth $h-1$ has not been identified, each of them is part of a possible spine. More formally, the set of possible spines $\mathfrak{S}_h$ is made of all the branches from the root to $v$ whenever $\mathcal{D}(v) = h-1$ and the type of $v$ has not been identified as normal. With this notation, if $\#\mathfrak{S}_h=1$, then the spine has been fully reconstructed. In all cases, $\bigcap_{\mathbf{s}\in\mathfrak{S}_h} \mathbf{s}$ is exactly the set of nodes identified as special, while the complement $\bigcup_{\mathbf{s}\in\mathfrak{S}_h}\mathbf{s}\setminus \bigcap_{\mathbf{s}\in\mathfrak{S}_h}\mathbf{s}$ is composed of all the nodes that can not been identified in an algorithmic way.

Spine candidates can be indexed by their first unobserved node. Given a node $v$ in $T$, the sequence of ancestors of $v$ is denoted $\mathcal{A}(v)$,
\begin{equation}\label{eq:def:A}
	\mathcal{A}(v) = \left(\mathcal{P}^h(v) , \mathcal{P}^{h-1}(v), \dots , \mathcal{P}(v) \right),
\end{equation}
where $\mathcal{P}(v)$ is the parent of $v$ in $T$ and recursively $\mathcal{P}^h(v) = \mathcal{P}(\mathcal{P}^{h-1}(v))$. If $\mathcal{D}(v)=h$, then $\mathcal{A}(v)$ is an element of $\mathfrak{S}_h$. In the whole paper, we identify (when there is no ambiguity) $\mathcal{A}(v)$ with the sequence of numbers of children along $\mathcal{A}(v)$, i.e. $(\#\mathcal{C}(u)\,:\,u\in\mathcal{A}(v))$.

\begin{figure}[t]
\centering
\includegraphics[width=0.48\textwidth]{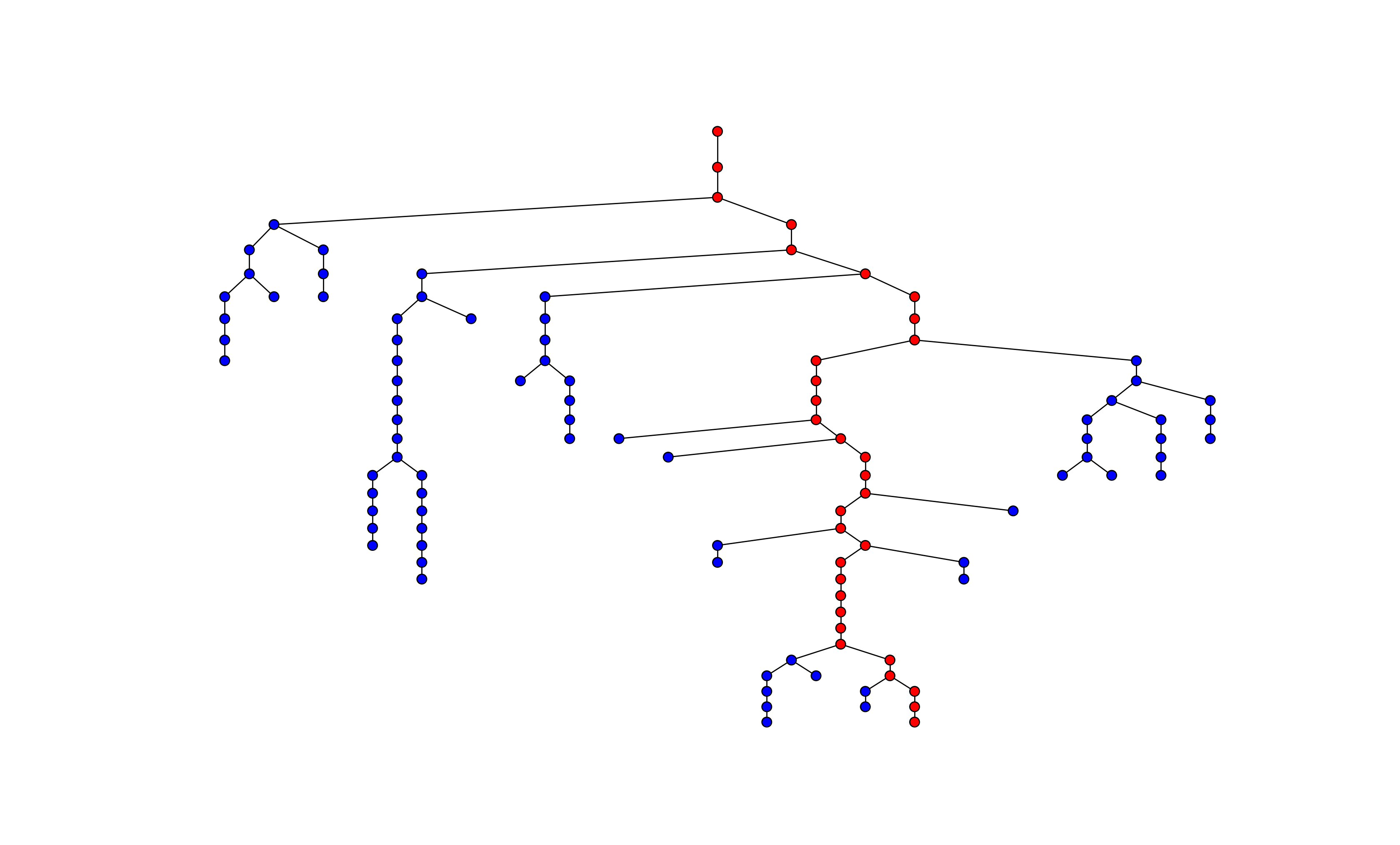}~
\includegraphics[width=0.48\textwidth]{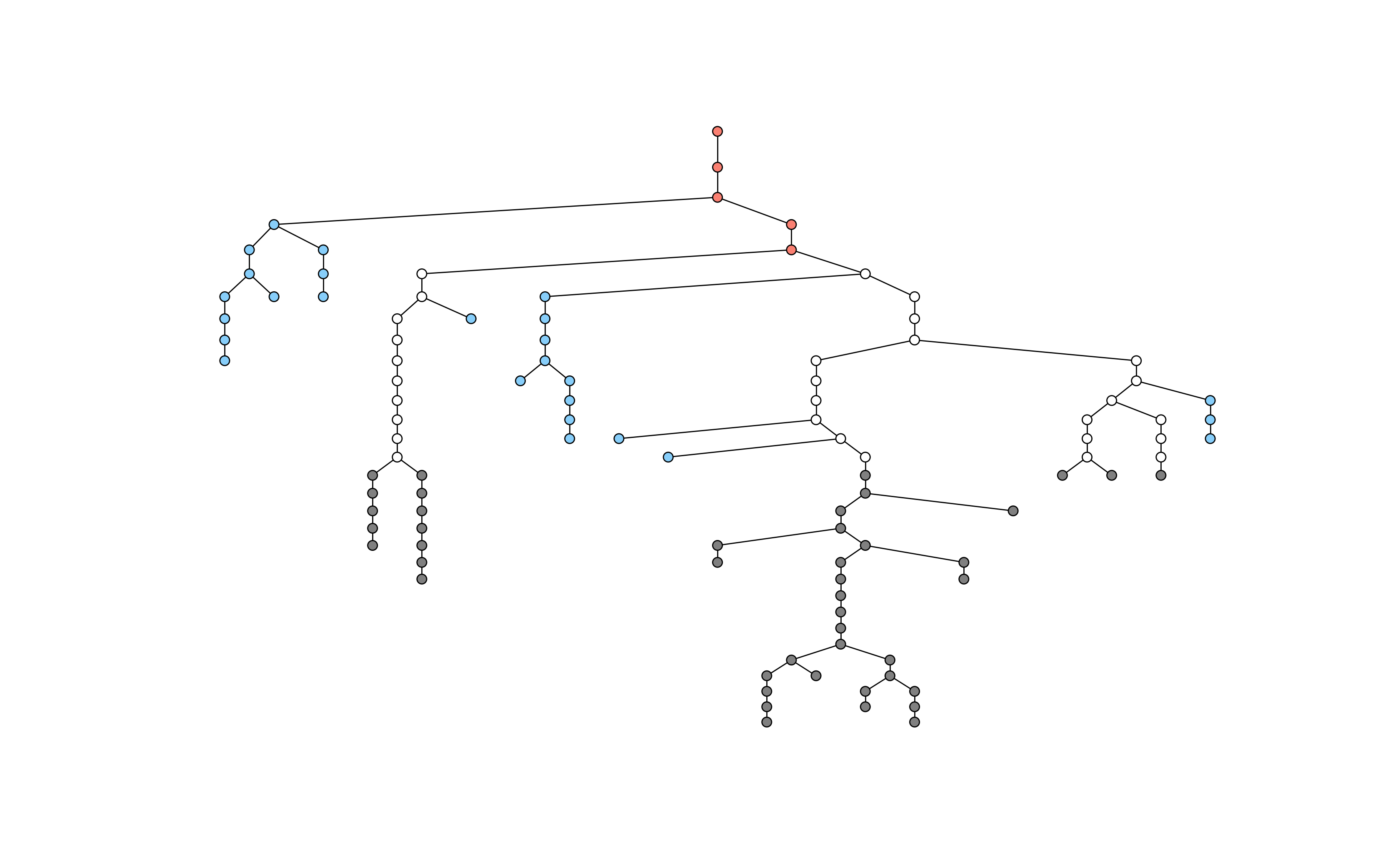}
\caption{A spinal-structured tree simulated until generation $30$ with normal nodes in blue and special nodes in red (left). We assume that it is observed until generation $h=15$ and identify the type of the nodes using Proposition~\ref{prop:id:spine} (right) with the following color code: light blue for identified normal nodes, light red for identified special nodes, gray for unobserved nodes, and white for unidentified types.}
\label{fig:ex:spine}
\end{figure}


\section{Ugly Duckling} \label{sec:MLS}

In this section, we aim to develop an estimation method for the unknown parameters $\mu$ and $f$ as well as the spine $\mathcal{S}$ of a spinal-structured tree observed until generation $h$. The algorithm presented below takes advantage of the specific behavior of spinal-structured trees. We also present our main result of convergence that holds whatever the growth regime of the normal population, i.e. whatever the value of $m(\mu)$.

\subsection{Estimation of $\mu$} \label{ss:1stmu}

As it was remarked in the introduction, maximizing $\mathcal{L}_h^{\text{\tiny{SST}}}$ with respect to $\mu$ does not require to observe the type of the nodes. Consequently, $f$ being unknown, we can still construct a first estimate of $\mu$ as
$$ \widehat{\mu}_h=\argmax_{\mu\in\measure}
\mathcal{L}_h^{\text{\tiny{SST}}} (\mu,f).$$
Standard calculus shows that
\begin{equation}
\widehat{\mu}_h = \left( \cfrac{1}{\card{T}_{h}}\sum_{v\in\tree_h}\mathbbm{1}_{\child{v}=i} \right)_{i\in\{0,\dots,N\}} .
\label{eq:estmu1}
\end{equation}
We can notice that the optimum in $f$ of $\mathcal{L}_h^{\text{\tiny{SST}}}$ depends on the unknown spine $\mathcal{S}$, and is thus of no use at this stage.

\subsection{Selection of the spine}
\label{ss:ugly}
In $\mathfrak{S}_h$, the spine $\mathcal{S}$ is the unique element whose component-wise distribution is $\nu$ defined from \eqref{eq:biased:distrib}. In that sense, finding $\mathcal{S}$ is a sample selection problem, where, however, the distribution at stake $\nu$ is unknown. Our approach consists in estimating the spine by the sample that differs the most from the expected behavior of a sample made of normal nodes.

However, it should be observed that $\spineset{h}$ consist of surviving lineages. Thus, $\mu$ \textit{is not} the component-wise distribution of the samples of normal nodes in $\spineset{h}$, and, as a consequence, is not the right distribution to be compared to. Identifying the law of $\s\in\spineset{h}$ can be done thanks to the so-called many-to-one formula presented in the following theorem (see \cite{lyons1995conceptual}).
\begin{thm}Let $G$ be a Galton-Watson tree with birth distribution $\mu$ and $h$ be an integer. Then, for any bounded measurable function $\varphi:\mathbb{R}^{h}\to\mathbb{R}$, we have
	\begin{equation}
		\label{eq:manytoone}
		\mathbb{E}\left[\sum_{\{u\in G\,:\,\mathcal{D}(u)=h\}}\varphi\left(\mathcal{A}(u)\right)\right]~=~\mean{\mu}^{h}\,\mathbb{E}\left[\varphi(X_{0},\ldots,X_{h-1})\right],
	\end{equation}
where $\mathcal{A}(u)$ is defined in \eqref{eq:def:A} and $X_{1},\ldots,X_{h-1}$ is an i.i.d.\ family of random variables with common distribution $\B{\mu}$, where the operator $\B{}$ is defined, for any $p\in\measure$, by
\begin{equation}
	\label{eq:biasOperator}
	\B{p}(i)=\cfrac{ip_{i}}{\mean{p}} ,\quad \forall\,i\in\{1,\ldots,\pmax\}.
\end{equation}
\end{thm}
\noindent
With this new information in hand, we can now define the estimate of the spine as the \textit{Ugly Duckling} in $\spineset{h}$,
\begin{equation} \label{eq:def:hatspine}
	\hatspine=\argmax_{\mathbf{s}\in\spineset{h}}\KL{\overline{\mathbf{s}}}{\B{\hatmu}},
\end{equation}
where $\overline{X}$ denotes the empirical measure associated to the vector $X$ and $\KL{p}{q}$ denotes the Kullback-Leibler divergence between distributions $p$ and $q$. In this formula, we compare $\overline{\mathbf{s}}$ to $\B{\hatmu}$, with $\hatmu$ given by \eqref{eq:estmu1}, because the true distribution $\mu$ is obviously unknown.

\begin{rem}
Another approach would have consisted in selecting the spine as the most likely sample under $\nu$, which is unknown but can be estimated from an estimate of $\mu$ (for example, $\hatmu$ defined in \eqref{eq:estmu1}) and an estimate of $f$. However, as explained in Subsection~\ref{ss:1stmu}, the optimum of $\mathcal{L}_h^{\text{\tiny{SST}}}$ in $f$ depends on the spine. As a consequence, this approach would have resulted in an iterative algorithm where $f$ is estimated from the spine, and conversely the spine from $f$, likely highly dependent on the initial value.
\end{rem}

\subsection{Correction of $\mu$ and estimation of $f$}
\label{ss:2stmu}
We can remark that the estimate \eqref{eq:estmu1} of $\mu$ is the empirical distribution of the numbers of children in the tree. However, the tree is made of $h$ special nodes that do not follow $\mu$ which biases the estimation. Now we know how to estimate the spine, i.e. the set of special nodes in the tree, we can take this into account and correct the estimator of $\mu$ as
\begin{equation}
	\label{eq:estmuf}
	\hatmu^\star(i)=\cfrac{1}{\card{T}_{h}-h}\sum_{v\in \tree_h\setminus\hatspine}\mathbbm{1}_{\child{v}=i}, \quad \forall\,i\in\{0,\ldots,\pmax\}.
\end{equation}
Then we can estimate $f$ by maximizing (under the constraint $\sum_{i=0}^N f(i)\hatmu^\star(i)=1$) $\mathcal{L}_h^{\text{\tiny{SST}}}$ where the unknown spine $\mathcal{S}$ has been replaced by $\widehat{\mathcal{S}}_h$, which results in
 \begin{equation}
 	\label{eq:estf}
 	\hatf(i)=\cfrac{1}{\hatmu^\star(i)h}\sum_{v\in\hatspine}\mathbbm{1}_{\child{v}=i}, \quad \forall\,i\in\{0,\ldots,\pmax\}.
 \end{equation}
It should be noted that, by construction, no node of the spine estimate has no child, which implies $\hatf(0)=0$.

\subsection{Theoretical results}
\label{ss:mainthm}

The purpose of this section is to study the behavior of the Ugly Duckling method for large observation windows, i.e. $h\to\infty$. The main difficulty arising in our problem is to recover a substantial part of the spine. Depending on the growth-rate of the population, this question takes different shapes. Indeed, the number of spine candidates $\#\spineset{h}$ is highly dependent on the growth-rate $\mean{\mu}$ of the normal population in the tree.

First, in the subcritical case $\mean{\mu}<1$, the trees of normal individuals grafted on the spine tends to become extinct. In other words, the set of spines $\spineset{h}$ is essentially reduced to $\spine$ or at least to small perturbations of $\spine$. Thus, a macroscopic part of the spine can be directly identified without further difficulties following the algorithm of Section~\ref{s:algo:spine}. The only point that needs clarification is that if the unidentified part of the spine is not large enough to perturb the estimation. In that case, we could not guarantee that our estimators are convergent. 

In the critical and supercritical cases, identifying the spine becomes substantially harder as the set $\spineset{h}$ may have a large size and contains potentially long lineages of non-special individuals. In particular, if the number of possible spines is large, one may observe that the empirical distribution of the number of children along some lineages $\mathbf{s}\in\spineset{h}$ may experience large deviations from its distribution, so that
\[
\KL{\overline{\mathbf{s}}}{\B{\mu}}\gg \KL{\overline{\spine}}{\B{\mu}}.
\]
In such a situation, one would not be able to distinguish which of $\mathbf{s}$ or $\spine$ is the spine.

It follows that our ability to identify the spine relies on a dissimilarity/population-growth trade-off:
\begin{itemize}
	\item On the one hand, if the growth-rate of the population is small, the number of possible spines is small and none of the normal spines largely deviate from its expected distribution. Thus, we can identify the sample with law $\nu$ even if laws $\B{\mu}$ and $\nu$ are similar (but without being too close).
	\item On the other hand, if the growth-rate is large (that is $\mean{\mu}\gg1$), then one may expect large deviation samples. In such situation, we would not be able to recover the spine unless distributions $\B{\mu}$ and $\nu$ are very different.
\end{itemize}
It happens that the good way to measure the dissimilarity between two distribution $p$ and $q$ in our context is given by the following divergence,
\begin{equation}
\label{eq:AHDivergence}
\AH{p}{q}=\!\!\!\inf_{\begin{subarray}~(x,y,z)\in\measure^{3}\\\quad\delta\geq0\end{subarray}}\!\left\{\KL{x}{p}\!+\!\KL{y}{q}\!+\!\delta\KL{z}{q}\Bigg|\KL{\cfrac{\delta z+x}{\delta+1}}{p}\!-\!\KL{\cfrac{\delta z+y}{\delta+1}}{p}\!\geq 0\right\}.
\end{equation}
This idea is summarized in the following theorem where our convergence criterion relies on a comparison between $\log\left(\mean{\mu}\right)$ and $\AH{\B{\mu}}{\nu}$.
\begin{thm}\label{thm:mainThm}
If $\log\left(\mean{\mu}\right)-\AH{\B{\mu}}{\nu}<0$, then the following convergences hold almost surely,
\[
\hatmu^\star\xrightarrow[h\to\infty]{}\mu
\]
and
\[
\hatf \xrightarrow[h\to\infty]{}f.
\]
In addition, an order $h$ of the spine is recovered, that is
\[
\cfrac{\card{\spine\cap\hatspine}}{h}\xrightarrow[h\to\infty]{}1
\]
almost surely.
\end{thm}


\section{Proof of Theorem~\ref{thm:mainThm} in the subcritical case}
\label{s:subproof}
In subcritical cases, note that the criteria of Theorem \ref{thm:mainThm} is always satisfied. In addition, it is important to note that the first step of our estimation procedure is in this case a dummy step as it has (essentially) no use  in the following steps. If $\mean{\mu}<1$, our estimation works as for large $h$, we can automatically identify an order $h$ of the special individuals as the lineages of the normal ones tend to become extinct. Thus, the main point to check in the proof of the subcritical case is that enough spine is directly identifiable. We directly give the proof of Theorem \ref{thm:mainThm} in this case.

\begin{proof}[Proof of Theorem \ref{thm:mainThm}, subcritical case $\mean{\mu}< 1$]
The key point is that normal Galton-Watson trees induced by special individuals are very unlikely to reach a large height. Their number being finite at each generation, very few of them reach height $h$. In particular, they would be rather recent subtrees.

Let us denote by $K_{h}$ the length of spine that can be algorithmically identified (using the procedure presented in Proposition~\ref{prop:id:spine}) when the spinal-structured tree is observed up to height $h$. Now, recalling that the spinal-structured tree $\tree$ can be constructed by grafting an i.i.d.\ family of Galton-Watson tree $(G_{i,j})_{i,j\geq1}$ on the spine, $K_{h}$ is given by
\[
K_{h}=\sup\left\{
1\leq n \leq h \mid \mathcal{H}\left(G_{i,j}\right)<h-i, \quad \forall\,i\in\{1,\ldots,n\},\quad \forall\,j\in\{1,\ldots,S_{i}-1\}\right\},
\]
where $S_{1},\ldots,S_{h}$ denote the numbers of special children of the individuals of the spine.
Thus,
\begin{align*}
\mathbb{P}\left(K_{h}\geq n\right)&=\mathbb{P}\Bigg(\bigcap_{i=1}^{n}\bigcap_{j=1}^{S_{i}-1}\left\{\mathcal{H}\left(G_{i,j}\right)<h-i\right\}\Bigg)\\
&=\prod_{i=1}^{n}\mathbb{E}\left[p^{S_{i}-1}_{h-i}\right].
\end{align*}
where $p_{l}$ denotes the probability that a tree of type $G_{i,j}$ becomes extinct before reaching height $l$. We then have that
\[
\mathbb{P}\left(K_{h}\geq n\right)\geq \prod_{i=1}^{n}p_{h-i}^{\mathbb{E}\left[S_{i}-1\right]}=\left(\prod_{i=1}^{n}p_{h-i}\right)^{\mean{\nu}-1}
\]
by Jensen's inequality. Fixing some $\ee>0$, we thus have that
\[
\mathbb{P}\left(1-\frac{K_{h}}{h}>\ee\right)=1-\mathbb{P}\left(K_{h}\geq \lfloor(1-\ee)h\rfloor\right)\leq 1-\left(\prod_{i=1}^{\lfloor(1-\ee)h\rfloor}p_{h-i}\right)^{\mean{\nu}-1}
\]
In the subcritical case, it is known \cite{AN} that
\[
p_{l}\geq1- \gamma^{l}
\]
for some real number $\gamma\in(0,1)$. Hence
\[
\mathbb{P}\left(1-\frac{K_{h}}{h}>\ee\right)\leq 1-\left(\prod_{i=1}^{\lfloor(1-\ee)h\rfloor}(1-\gamma^{h-i})\right)^{\mean{\nu}-1}\leq1-\left(1-\gamma^{\ee h}\right)^{(\mean{\nu}-1)\lfloor(1-\ee)h\rfloor}.
\]
It is then easily checked that
\[
\sum_{h\geq 1}\left(1-\left(1-\gamma^{\ee h}\right)^{(\mean{\nu}-1)\lfloor(1-\ee)h\rfloor}\right)<\infty
\]
which entails, through Borel-Cantelli lemma, the almost sure convergence of $\frac{K_{h}}{h}$ toward $1$.
Now, as we almost surely have $K_{h}\leq \card{\hatspine\cap\spine}$, the convergence of $\hatmu$ and $\hatmu^\star$ follows closely the proof of the forthcoming Proposition \ref{prop:CVmuhat}, while the convergence of $\hatf$ can be easily deduced from the Law of Large Numbers.
\end{proof}


\section{On the rate function of large deviations in sample selection}
\label{s:LDP}
\newcommand{\citopt}[1]{{\renewcommand\paramepsilon{#1}\eqref{eq:optimproblem}}}
In Lemma \ref{lem:pertMultEstimate} in the forthcoming Section \ref{sec:mainproof}, we show a large deviation-type estimate for the probability that the empirical distribution of some branch of the spinal-structured tree is closer from $\nu$ than the one of the true spine (in Kullback-Leibler divergence). The purpose of this section is to study the rate function of this estimate, and is a preliminary of the proof of Theorem \ref{thm:mainThm} in the critical and supercritical cases. In this whole section, we choose some distribution $p$ and $q$ in $\measure$ such that $p\neq q$. Our goal is to study the following parametric optimization problem referenced as problem \eqref{eq:optimproblem}.
\begin{OptimEquation}
\begin{array}{>{(P_{\paramalpha,\paramepsilon}(\stepcounter{compteleslignes}\thecompteleslignes))}clll}
    ~&\text{min.}	& \displaystyle f_{\alpha}(\x{1},\x{2},\x{3}) \!=\!(1-\alpha)\big[\KL{\x{1}}{p}+\KL{\x{2}}{q}\big]+\alpha \KL{\x{3}}{q} \\
    ~&\text{s.t.}	& \displaystyle\x{j}_{i} \geq 0,\qquad\forall\,(i,j)\in\{0,\ldots,\pmax\}\times \{1,2,3\}, \\
    ~& 		& \displaystyle g_{j}(\x{1},\x{2},\x{3})=\sum_{i=0}^{\pmax}\mathbf{x}_{i}^{(j)}-1=0,\qquad\forall\,j\in\{1,2,3\},\\
    ~&		& \displaystyle H_{\alpha,\ee}(\x{1},\x{2},\x{3})\geq0, &
  \end{array}
\label{eq:optimproblem}
\end{OptimEquation}
\noindent
where
$$ 
H_{\alpha,\ee}(\x{1},\x{2},\x{3})=\KL{(1-\alpha)\x{1}+\alpha\x{3}}{p}-\KL{(1-\alpha)\x{2}+\alpha\x{3}}{p}+\ee.
$$
The value function associated to problem \eqref{eq:optimproblem} is denoted $V:[0,1]^{2}\ni(\alpha,\ee)\mapsto V(\alpha,\ee)\in\mathbb{R}_{+}$ and is given by
\begin{equation}
\label{eq:valueF}
V(\alpha,\ee)=\inf_{(x,y,z)\in\measure^{3}}\left\{(1-\alpha)\left(\KL{x}{p}+\KL{y}{q}\right)+\alpha \KL{z}{q}~\big|~H_{\alpha,\ee}(x,y,z) \geq 0\right\}.
\end{equation}
In the particular situation where $\ee=0$, the value function associated to problem  {\renewcommand\paramepsilon{0}\eqref{eq:optimproblem}} is denoted $v:[0,1]\ni\alpha\mapsto v(\alpha)\in\mathbb{R}_{+}$. Our goal is to show the following theorem.
\begin{thm}
\label{thm:optim}
The value function $V$ is continuous. In addition, for any $\rho \in(0,1)$, there exists $\ee^{\ast}>0$ such that
\[
V(\alpha,\ee)\geq v(\alpha)-\rho, \quad \forall\,\alpha\in[0,1],\quad \forall\,\ee\in[0,\ee^{\ast}],
\]
and
\[
\frac{v(\alpha)}{1-\alpha}\xrightarrow[\alpha \to 1]{}\bah{p}{q},
\]
where $d_{B}$ is the Bhattacharyya divergence defined by
\begin{equation}
\label{eq:bah}
\bah{p}{q}=-2\log\left(\sum_{i=1}^{\pmax}\sqrt{p_{i}q_{i}}\right)
\end{equation}
\end{thm}

\noindent
To show this result, we begin by defining the parameters dependent Lagrangian associated with problem {\eqref{eq:optimproblem}} by 
\begin{eqnarray*}
	&&L(\x{1},\x{2},\x{3},w,u,\gamma,\alpha,\ee)\\
	&=&(1-\alpha)\left(\KL{\x{1}}{p}+\KL{\x{2}}{q}\right)+\alpha \KL{\x{3}}{q}\\
	&+&\sum_{i=1}^{\pmax}\sum_{j=1}^{3}w_{i,j}\mathbf{x}^{(j)}_{i}+\gamma\left( \KL{(1-\alpha)\x{1}+\alpha \x{3}}{p}-\KL{(1-\alpha)\x{2}+\alpha \x{3}}{p}+\ee\right) \\
	&+&\sum_{j=1}^{3}u_{j}\sum_{i=1}^{\pmax}(\mathbf{x}^{(j)}_{i}-1),
\end{eqnarray*}
where $\gamma,u_{1}$, $u_{2}$, $u_{3}$, $(w_{i,j})_{1\leq i\leq N,1\leq j \leq 3}$ are the Lagrange multipliers. Thus, the first order optimality conditions are given by
\begin{subequations}
	\label{eq:firstOrder}
	\begin{empheq}[left=\empheqlbrace]{align}
		&(1-\alpha)\left\{\log\left(\cfrac{\x{1}_{i}}{p_{i}}\right)+1\right\}+\gamma(1-\alpha)\left\{\log\left(\cfrac{(1-\alpha) \x{1}_{i}+\alpha \x{3}_{i}}{p_{i}}\right)+1\right\}+\lambda=0,\quad \forall\,i \in\llbracket 0,\pmax\rrbracket ,\\
		&(1-\alpha)\left\{\log\left(\cfrac{\x{2}_{i}}{q_{i}}\right)+1\right\}-\gamma(1-\alpha)\left\{\log\left(\cfrac{(1-\alpha) \x{2}_{i}+\alpha\x{3}_{i}}{p_{i}}\right)+1\right\}+\mu=0,\quad \forall\,i \in\llbracket 0,\pmax\rrbracket ,\\
		&\alpha\left\{\log\left(\cfrac{\x{3}_{i}}{q_{i}}\right)+1\right\}+\alpha\gamma\left\{\log\left(\cfrac{(1-\alpha) \x{1}_{i}+\alpha\x{3}_{i}}{(1-\alpha) \x{2}_{i}+\alpha\x{3}_{i}}\right)\right\}+\nu=0,\quad \forall\,i \in\llbracket 0,\pmax\rrbracket,\\
		&\gamma\Bigg(\KL{(1-\alpha)\x{1}+\alpha \x{3}}{p}-\KL{(1-\alpha)\x{2}+\alpha \x{3}}{p}\Bigg)=0,
	\end{empheq}
\end{subequations}
where $\lambda$, $\mu$, $\nu$ are the Lagrange multipliers associated with the constraints {\renewcommand\numligne{(3)}\eqref{eq:optimproblem}} (corresponding to $u$ in the definition of the Lagrangian).

Let us point out that these optimality conditions do not hold for feasible points such that $\x{i}_{j}=0$ for some $i$ and $j$ because our problem is not smooth at these points. It only holds
for feasible points in the interior of $\mathbb{R}_{+}^{3(\pmax+1)}$. In Lemma \ref{lem:nozero}, we show that there is no optimal solution in the boundary of $\mathbb{R}_{+}^{3(\pmax+1)}$ which justifies the use of conditions \eqref{eq:firstOrder}.
The set of Lagrange multipliers associated with a feasible point $(x,y,z)$ is denoted $\lag(x,y,z)$ (and is a subset of $\mathbb{R}_{-}^{3(\pmax+1)}\times\mathbb{R}^{3}\times \mathbb{R}_{-}$). In particular, let us highlight that due to the inequality constraint {\renewcommand\numligne{(4)}\eqref{eq:optimproblem}}, we require $\gamma\leq 0$. We denote by $\solopt(\alpha,\ee)$ the set of solutions of the above problem for given parameters $(\alpha,\ee)$ and $\feas(\alpha,\ee)$ the set of feasible points. In the particular case where $\ee=0$, we use the notations $\solopt(\alpha)$ and $\feas(\alpha)$ for $\solopt(\alpha,0)$ and $\feas(\alpha,0)$ respectively. Our first goal is to show that, for any $(\x{1},\x{2},\x{3})\in \solopt(\alpha,\ee)$, we have $\mathbf{x}^{(j)}_{i}>0$ for all $i,j$. This is the point of the following lemma.
\begin{lem}
	\label{lem:nozero}
	Consider the set $\solopt(\alpha,\ee)$ of solutions of problem \eqref{eq:optimproblem}. Then, for $\ee$ small enough and any $\alpha\in(0,1)$, we have
	\[
	\solopt(\alpha,\ee)\cap \partial\mathbb{R}_{+}^{3(\pmax+1)}=\emptyset.
	\] 
\end{lem}
\begin{proof}
The proof has been deferred into Appendix~\ref{app:proof:nozero}.
\end{proof}

\begin{rem}
	\label{eq:zero}
In the cases where $\ee=0$, note that one can easily check using the first order optimally conditions that the inequality constraint {\renewcommand\numligne{(4)}\eqref{eq:optimproblem}} is always saturated. Thus, in the following, we will always assume that $\gamma<0$.
\end{rem}
\begin{proof}[Proof of Theorem \ref{thm:optim}]~

\medskip

\noindent
{\bf Step 1: Solving {\renewcommand\paramalpha{0}\renewcommand\paramepsilon{0}\eqref{eq:optimproblem}}, i.e. $\alpha=\ee=0$}

\noindent
In this case, the first order optimality conditions \eqref{eq:firstOrder} become
\begin{subequations}
	\begin{empheq}[left=\empheqlbrace]{align}
		&\log\left(\cfrac{x^{(0)}_{i}}{p_{i}}\right)+1+\gamma\left\{\log\left(\cfrac{ x^{(0)}_{i}}{p_{i}}\right)+1\right\}+\lambda=0,\quad \forall\,i \in\llbracket 0,\pmax\rrbracket, \label{eq:Faz1}\\
		&\log\left(\cfrac{y^{(0)}_{i}}{q_{i}}\right)+1-\gamma\left\{\log\left(\frac{y^{(0)}_{i}}{p_{i}}\right)+1\right\}+\mu=0,\quad \forall\,i \in\llbracket 0,\pmax\rrbracket, \label{eq:Faz2}\\
		&\gamma\left(\KL{x^{(0)}}{p}-\KL{y^{(0)}}{p}\right)=0,\label{eq:Faz3}\\
		&\sum_{i=0}^{\pmax} x^{(0)}_{i}=\sum_{i=0}^{\pmax} y^{(0)}_{i}=1.\label{eq:Faz4}
	\end{empheq}
\end{subequations}

\noindent
If we assume that $\gamma\neq -1$, eqs.\,\eqref{eq:Faz1}, \eqref{eq:Faz3}, and \eqref{eq:Faz4} lead to
\[
x^{(0)}_{i}=y^{(0)}_{i}=p_{i},\quad \forall\,i \in \llbracket 0,\pmax\rrbracket,
\]
which is not compatible with eq.\,\eqref{eq:Faz2} unless $p=q$. In addition, $\gamma=0$ leads to $x^{(0)}=p=$ and $y^{(0)}=q$ which is easily checked to be not feasible.
Thus, we have $\gamma= -1$ and eq.\,\eqref{eq:Faz2} then gives
\[
y^{(0)}_{i}e^{\mu/2}=\sqrt{p_{i}q_{i}},\quad \forall\,i\in \llbracket 0,\pmax\rrbracket
\]
which gives, using eq.\,\eqref{eq:Faz4},
\[
y^{(0)}_{i}=\cfrac{\sqrt{p_{i}q_{i}}}{\sum_{l=0}^{\pmax}\sqrt{p_{l}q_{l}}},\quad \forall\,i\in \llbracket 0,\pmax\rrbracket.
\]
It follows from eq.\,\eqref{eq:Faz3}, that $(x^{(0)},y^{(0)},z^{(0)})$ with
\[
\begin{cases}
	x^{(0)}_{i}=y^{(0)}_{i}=\cfrac{\sqrt{p_{i}q_{i}}}{\sum_{l=0}^{\pmax}\sqrt{p_{l}q_{l}}},\quad \forall\,i \in \llbracket 0,\pmax\rrbracket,\\
	z^{(0)}=q,
\end{cases}
\]
is a feasible optimal solution of problem ${\renewcommand\paramepsilon{0}\renewcommand\paramalpha{0}\eqref{eq:optimproblem}}$. In particular,
\begin{align*}
	f_{0}(x^{(0)},y^{(0)},z^{(0)})&=\sum_{i=0}^{\pmax}\cfrac{\sqrt{p_{i}q_{i}}}{\sum_{l=0}^{\pmax}\sqrt{p_{i}q_{i}}}\log\left(\cfrac{\sqrt{p_{i}q_{i}}}{p_{i}\sum_{l=0}^{\pmax}\sqrt{p_{i}q_{i}} }\right)+\sum_{i=0}^{\pmax}\cfrac{\sqrt{p_{i}q_{i}}}{\sum_{l=0}^{\pmax}\sqrt{p_{i}q_{i}}}\log\left(\cfrac{\sqrt{p_{i}q_{i}}}{q_{i}\sum_{l=0}^{\pmax}\sqrt{p_{i}q_{i}} }\right)\\
	&=-2\log\left(\sum_{l=0}^{\pmax}\sqrt{p_{i}q_{i}} \right)=\bah{p}{q},
\end{align*}
where $\bah{\cdot}{\cdot}$ is the Bhattacharyya divergence defined in eq.\,\eqref{eq:bah}.

\noindent
{\bf Step 2: Continuity of the value function}

\noindent
The goal of this step is to show that the full value function $V$ is continuous. To do so, we apply Theorem 2.1 in conjunction to Theorem 2.8 of \cite{fiacco}. In view of this theorems, the only point that needs clarification is that
\[
\overline{\left\{(x,y,z)\in \measure^{3} \mid H_{\alpha,\ee}(x,y,z)>0\right\}}=\left\{(x,y,z)\in \measure^{3} \mid H_{\alpha,\ee}(x,y,z)\geq0\right\}.
\]
To do so, it suffices to show that for any $(a,b,c)\in \measure^{3}$ such that $H_{\alpha,\ee}(a,b,c)=0$ and any $\delta>0$, there exists an element $(\tilde{a},\tilde{b},\tilde{c})\in \measure^{3}$ such that $H_{\alpha,\ee}(\tilde{a},\tilde{b},\tilde{c})>0$ with
\[
\|(a,b,c)-(\tilde{a},\tilde{b},\tilde{c})\|_1<\delta.
\]
As the proof of this follows closely the ideas of the proof of Lemma \ref{lem:nozero}, we do not write down the details. It then implies that $V$ is continuous. The first statement of Theorem \ref{thm:optim} now follows from the compactness of $[0,1]$.

\noindent
{\bf Step 3: Limits of $\frac{v(\alpha)}{1-\alpha}$ as $\alpha\to1$}

\noindent
For any $\alpha \in[0,1)$, it is easily seen that
\begin{eqnarray*}
&&\cfrac{v(\alpha)}{(1-\alpha)}\\
&=&\!\!\!\!\!\!\!\inf_{(x,y,z)\in \measure^{3}}\!\!\left\{\KL{x}{p}+\KL{y}{q}+\frac{\alpha\,\KL{z}{q}}{1-\alpha}\Bigg| \KL{x+\cfrac{\alpha}{1-\alpha}z}{p}-\KL{y+\cfrac{\alpha}{1-\alpha}z}{p}\geq 0 \right\}.
\end{eqnarray*}
It is equivalent to study the behavior of 
\begin{equation}
	\label{eq:delta}
\mathcal{V}(\delta)=\inf_{(x,y,z)\in \measure^{3}} \left\{\KL{x}{p}+\KL{y}{q}+\delta\KL{z}{q}\Bigg| \KL{x+\delta z}{p}-\KL{y+\delta z}{p}\geq 0 \right\},
\end{equation}
as $\delta$ goes to infinity.
So, let $(\delta_n)_{n\geq 1}$ be some sequence of real numbers such that $\delta_n\xrightarrow[n\to \infty]{}\infty$, and set
\[
\solopt_n:=\solopt\left(\frac{\delta_n}{1+\delta_n},0\right).
\]
Now, for all $n\geq 1$, choose $(x^{(n)},y^{(n)},z^{(n)})\in \solopt_n$.
As $\cup_{n\geq 1}\solopt_{n}\subset \measure$ is relatively compact, we may assume, extracting a sub-sequence if needed, that $(x^{(n)},y^{(n)},z^{(n)})$ converges to some element $(x^{\ast},y^{\ast},z^{\ast})\in \measure^{3}$. Now, assume that
\[
\lim_{n\to \infty}\|z^{(n)}-q\|_{1}>0.
\]
However, this would imply that $\liminf_{n\to\infty}\KL{z^{(n)}}{q}>0$, and thus that
\[
\liminf_{n\to \infty}\left\{ \KL{x^{(n)}}{p}+\KL{y^{(n)}}{q}+\delta_{n}\KL{z^{(n)}}{q}\right\}\geq \liminf_{n\to \infty}\delta_{n}\KL{z^{(n)}}{q}=\infty,
\]
but this is impossible, since, according to Step $1$, $\mathcal{V}(\delta)\leq \bah{p}{q}$ (because the solution given in Step $1$ is always feasible).
It follows that
\[
\lim_{n\to \infty}\|z^{(n)}-q\|_{1}=0
\]
and $z^{\ast}=q$.
Now, for fixed $n\geq 1$, the first order optimality conditions of problem \eqref{eq:delta} take the form
\begin{equation}
\label{eq:firstOpt2}
\left\{
\begin{aligned}
	&\log\left(\cfrac{x^{(n)}_{i}}{p_{i}} \right)+1+\gamma_n\left(\log\left(\cfrac{\delta_nz^{(n)}_{i}+x^{(n)}_{i}}{p_{i}} \right)+1\right)  +\lambda_n=0\\
	&\log\left(\cfrac{y^{(n)}_{i}}{q_{i}} \right)+1-\gamma_n\left(\log\left(\cfrac{\delta_nz^{(n)}_{i}+y^{(n)}_{i}}{p_{i}} \right)+1\right)  +\mu_n=0\\
	&\delta_n\log\left(\cfrac{z^{(n)}_{i}}{q_{i}} \right)+\delta_n+\gamma_n\delta_n\left(\log\left(\cfrac{\delta_nz^{(n)}_{i}+x^{(n)}_{i}}{\delta_nz^{(n)}_{i}+y^{(n)}_{i}} \right)\right)  +\nu_n=0,
\end{aligned}
\right.
\end{equation}
for some Lagrange multipliers $(\lambda_{n},\mu_{n},\nu_{n},\gamma_{n})\in\lag(x^{(n)},y^{(n)},z^{(n)})$.

We now show that the sequence $\gamma_{n}$ must be bounded. For this, let us assume that $\gamma_{n}$ is unbounded. So, extracting a subsequence if needed, we may assume that
 \[
 \gamma_{n}\xrightarrow[n\to\infty]{}-\infty.
 \]
Second equation of \eqref{eq:firstOpt2} implies that, for all $i$,
\[
\cfrac{y_{i}^{(n)}p_{i}}{q_{i}}=p_{i}\left(\cfrac{\delta_{n}z^{(n)}_{i}+y_{i}^{(n)}}{p_{i}}\right)^{\gamma_{n}}\exp\left(-1+\gamma_{n}-\mu_{n}\right).
\]
Summing over $i$ and using Jensen's inequality (since $\gamma_{n}<0$), we obtain
\[
\sum_{i=1}^{\pmax}\cfrac{y_{i}^{(n)}p_{i}}{q_{i}}=\sum_{i=1}^{\pmax}p_{i}\left(\cfrac{\delta_{n}z^{(n)}_{i}+x_{i}^{(n)}}{p_{i}}\right)^{\gamma_{n}}\exp\left(-1+\gamma_{n}-\mu_{n}\right)\geq(1+\delta_{n})^{\gamma_{n}}\exp\left(-1+\gamma_{n}-\mu_{n}\right).
\]
Thus,
\[
\log\left(\sum_{i=1}^{\pmax}\cfrac{y_{i}^{(n)}p_{i}}{q_{i}}\right)-\gamma_{n}\log(1+\delta_{n})+1-\gamma_{n}+\mu_{n}\geq 0.
\]
Now, eqs.\,\eqref{eq:firstOpt2} also give, for all $1\leq i\leq \pmax$,
\begin{eqnarray}
&&\log\left(\cfrac{x^{(n)}_{i}}{p_{i}} \right)+\gamma_n\log\left(\cfrac{z^{(n)}_{i}+\delta_{n}^{-1}x^{(n)}_{i}}{p_{i}} \right)+\gamma_n\log(\delta_n) +1+\gamma_{n}+\lambda_n \nonumber\\
&=&\log\left(\cfrac{x^{(n)}_{i}}{p_{i}} \right)+\gamma_n\log\left(\cfrac{\delta_{n}z^{(n)}_{i}+x^{(n)}_{i}}{(\delta_{n}+1)p_{i}} \right)+\gamma_n\log(1+\delta_n) +1+\gamma_{n}+\lambda_n=0. \label{eq:op1}
\end{eqnarray}
and 
\begin{equation*}
\log\left(\cfrac{y^{(n)}_{i}}{q_{i}} \right)-\gamma_n\log\left(\cfrac{\delta_{n}z^{(n)}_{i}+y^{(n)}_{i}}{(\delta_{n}+1)p_{i}} \right)-\gamma_n\log(1+\delta_n) +1+\gamma_{n}+\mu_n=0.
\end{equation*}
Hence, for $n$ large enough, we have
\[
\log\left(\cfrac{y^{(n)}_{i}}{q_{i}} \right)-\gamma_n\log\left(\cfrac{\delta_{n}z^{(n)}_{i}+y^{(n)}_{i}}{(\delta_{n}+1)p_{i}} \right)-\log\left(\sum_{i=1}^{\pmax}\cfrac{y_{i}^{(n)}p_{i}}{q_{i}}\right)\leq 0,
\]
In particular, this implies that $y^{\ast}_{i}=0$ for all $i$ such that $q_{i}> p_{i}$. 
Similarly with \eqref{eq:op1}, one gets
\[
\log\left(\cfrac{x^{(n)}_{i}}{p_{i}} \right)+\gamma_n\log\left(\cfrac{\delta_{n}z^{(n)}_{i}+x^{(n)}_{i}}{(\delta_{n}+1)p_{i}} \right)\leq 0,
\]
and $x^{\ast}_{i}=0$ for all $i$ such that $p_{i}> q_{i}$.
Now, a direct computation gives
\[
\KL{x^{\ast}}{p}+\KL{y^{\ast}}{q}\geq -\log\left(\sum_{j\in J}p_{i}\right)-\log\left(\sum_{i\in I}q_{i}\right)
\]
with $I=\{i\in\{1,\ldots,\pmax\}\mid q_{i}\leq p_{i}\}$ and $J=\{i\in\{1,\ldots,\pmax\}\mid p_{i}\leq q_{i}\}$. However,
\[
\left(\sum_{i\in J }p_{i}\right)\left(\sum_{i\in I }q_{i}\right)\leq\left( \sum_{i\in J }\sqrt{p_{i}q_{i}}\right)\left( \sum_{i\in I }\sqrt{p_{i}q_{i}}\right)<\left( \sum_{i=1 }^{\pmax}\sqrt{p_{i}q_{i}}\right)^{2}.
\]
Thus,
\[
 -\log\left(\sum_{j\in J}p_{i}\right)-\log\left(\sum_{i\in I}q_{i}\right)>\bah{p}{q}.
\]
But, as the solution of Step 1 is always an admissible solution, this is absurd since this would implies that, for $n$ large enough, $(x^{(n)},y^{(n)},z^{(n)})$ is not optimal. From what precedes, we conclude that $\gamma_{n}$ is bounded. Thus, extracting again a subsequence if needed, we can suppose that there exists some $\gamma_{\infty}\leq 0$ such that
\[
\gamma_{n}\xrightarrow[n\to\infty]{} \gamma_{\infty}.
\]
In addition, eq.\,\eqref{eq:op1} implies that the sequence $\left(\gamma_n\log(\delta_n) +\lambda_n\right)_{n\geq 1}$ is bounded as well (because there must be at least one $i$ such that $\lim_{n\to\infty }x^{(n)}_{i}>0$), which we may also assume to be convergent. 
From these and eq.\,\eqref{eq:op1}, it follows that
\begin{equation}
\label{eq:xn}
x_{i}^{(n)}=p_{i}^{1+\gamma_{n}}q_{i}^{-\gamma_{n}}e^{c_{n}+\gamma_{n}\mathcal{O}(\delta_{n}^{-1})}=\cfrac{p_{i}^{1+\gamma_{\infty}}q_{i}^{-\gamma_{\infty}}}{\sum_{j=1}^{\pmax}p_{j}^{1+\gamma_{\infty}}q_{j}^{-\gamma_{\infty}}}+o(1), \quad \forall\,1\leq i\leq n,
\end{equation}
and similarly, we have
\begin{equation}
\label{eq:yn}
y_{i}^{(n)}=p_{i}^{-\gamma_{n}}q_{i}^{1-\gamma_{n}}e^{\tilde{c}_{n}+\gamma_{n}\mathcal{O}(\delta_{n}^{-1})}=\cfrac{p_{i}^{-\gamma_{\infty}}q_{i}^{1+\gamma_{\infty}}}{\sum_{j=1}^{\pmax}p_{j}^{-\gamma_{\infty}}q_{j}^{1+\gamma_{\infty}}}+o(1),\quad \forall\,1\leq i\leq n,
\end{equation}
where $(c_{n})_{n\geq 1}$ and $(\tilde{c_{n}})_{n\geq 1}$ are some convergent sequences. Denoting
\[
h(\gamma)=\left(\cfrac{p_{i}^{1+\gamma}q_{i}^{-\gamma}}{\sum_{j=1}^{\pmax}p_{j}^{1+\gamma}q_{j}^{-\gamma}}\right)_{1\leq i\leq \pmax},
\]
it follows, setting for any $x\in\measure$, $\auxfonc(x)=\KL{x}{p}$, that (see Remark \ref{eq:zero})
\begin{eqnarray*}
	0&=&\auxfonc\left(\alpha_{n} z^{(n)}+(1-\alpha_{n})x^{(n)}\right)-\auxfonc\left(\alpha_{n} z^{(n)}+(1-\alpha_{n})y^{(n)}\right)\\
	&=&\auxfonc\left(\alpha_{n}q+(1-\alpha_{n})h(\gamma_{n})\right)-\auxfonc\left(\alpha_{n}q+(1-\alpha_{n})h(-1-\gamma_{n})\right)\\
	&+&\nabla \auxfonc\left(\alpha_{n}q+(1-\alpha_{n})h(\gamma_{n})\right)\cdot\left(\alpha_{n}(z^{(n)}-q)+(1-\alpha_{n})(x^{(n)}-h(\gamma_{n}))\right)\\
	&-&\nabla \auxfonc\left(\alpha_{n}q+(1-\alpha_{n})h(-1-\gamma_{n})\right)\cdot\left(\alpha_{n}(z^{(n)}-q)+(1-\alpha_{n})(y^{(n)}-h(-1-\gamma_{n}))\right)+\mathcal{O}\left(\frac{1}{\delta_{n}^{2}}\right)\\
	&=&\auxfonc\left(\alpha_{n}q+(1-\alpha_{n})h(\gamma_{n})\right)-\auxfonc\left(\alpha_{n}q+(1-\alpha_{n})h(-1-\gamma_{n})\right)\\
	&+&\alpha_{n}\left(\nabla \auxfonc\left(\alpha_{n}q+(1-\alpha_{n})h(\gamma_{n})\right)-\nabla \auxfonc\left(\alpha_{n}q+(1-\alpha_{n})h(-1-\gamma_{n})\right)\right)+o\left(\frac{1}{\delta_{n}}\right),
\end{eqnarray*}
but since $\nabla \auxfonc$ exists and is continuous in a neighborhood of $q$, we get
\begin{eqnarray*}
	&&\auxfonc\left(\alpha_{n} z^{(n)}+(1-\alpha_{n})x^{(n)}\right)-\auxfonc\left(\alpha_{n} z^{(n)}+(1-\alpha_{n})y^{(n)}\right)\\
	&=&\auxfonc\left(\alpha_{n}q+(1-\alpha_{n})h(\gamma_{n})\right)-\auxfonc\left(\alpha_{n}q+(1-\alpha_{n})h(-1-\gamma_{n})\right)+o\left(\frac{1}{\delta_{n}}\right)\\
	&=&(1-\alpha_{n})\nabla \auxfonc(\alpha_{n}q)\cdot\left(h(\gamma_{n})-h(-1-\gamma_{n})\right)+o\left(\frac{1}{\delta_{n}}\right)\\
	&=&0.
\end{eqnarray*}
Finally, as $1-\alpha_{n}\sim \frac{1}{\delta_{n}}$, it follows that
\[
\nabla K(q)\cdot(h(\gamma_{\infty})-h(1-\gamma_{\infty}))=0.
\]
Now, since $\nabla K(q)=\left(\log(q_{i}/p_{i})+1\right)_{1\leq i\leq N}$, we have 
\begin{equation}
	\label{eq:RevEq1}
	\sum_{i=1}^{N}\log\left(\frac{q_{i}}{p_{i}}\right)\left(h_{i}(\gamma_{\infty})-h_{i}(-1-\gamma_{\infty})\right)=0.
\end{equation}
Since $p_{i}^{1+\gamma}q_{i}^{-\gamma}=(p_{i}/q_{i})^{-\gamma-1/2}(p_{i}q_{i})^{1/2}$, eq.\,\eqref{eq:RevEq1} can be written as
\begin{equation}
	\label{eq:rev2}
	F(-\gamma_{\infty}-1/2)=F(\gamma_{\infty}+1/2),
\end{equation}
where $F(\gamma)=F_{1}(\gamma)/F_{0}(\gamma)$ with, for any $k\in\mathbb{N}$,
\[
F_{k}(\gamma):=\sum_{i=1}^{N}\log^{k}\left(\frac{q_{i}}{p_{i}}\right)\cdot \left(\frac{q_{i}}{p_{i}}\right)^{\gamma}\sqrt{p_{i}q_{i}}.
\]
Let $\overline{\gamma_{\infty}}:=-\gamma_{\infty}-1/2$ which implies that eq.\,\eqref{eq:rev2} can be rewritten as $F(\overline{\gamma}_{\infty})=F(-\overline{\gamma}_{\infty})$. We shall show that the only solution of this equation is $\overline{\gamma}_{\infty}=0$. One can see that $F(1/2)=d_{KL}(q,p)>0$ and $F(-1/2)=-d_{KL}(p,q)<0$. Since $F'_{k}=F_{k+1}$, we obtain that 
\[
F'=\frac{F_{1}'F_{0}-F_{1}F_{0}'}{F_{0}^{2}}=\frac{F_{2}F_{0}-F_{1}^{2}}{F_{0}^2}>0,
\]
by the Cauchy-Schwarz inequality because $\log(q_{i}/p_{i})$ is not constant by $p\neq q$. Hence, $F$ is a strictly increasing function which implies that the only solution is $\overline{\gamma}_{\infty}=0$. Hence, $\gamma_{\infty}=1/2$ which finally gives, according to eqs.\,\eqref{eq:xn} and \eqref{eq:yn} that
\[
x^{\ast}_{i}=y^{\ast}_{i}=\cfrac{\sqrt{p_{i}q_{i}}}{\sum_{j=1}^{\pmax}\sqrt{p_{j}q_{j}}},\quad \forall\,1\leq i\leq \pmax.
\]
From this, it follows that
\[
\mathcal{V}(\delta_{n})\xrightarrow[n\to\infty]{}\bah{p}{q},
\]
which implies, since the sequence $(\delta_{n})_{n\geq 1}$ is arbitrary, that
\[
\lim_{\alpha\to 1}\cfrac{v(\alpha)}{1-\alpha}=\lim_{\delta\to\infty}\mathcal{V}(\delta)=\bah{p}{q}.
\]
This ends the proof.
\end{proof}

\section{Proof of Theorem \ref{thm:mainThm} in the critical and supercritical cases}
\label{sec:mainproof}
The purpose of this section is to prove Theorem \ref{thm:mainThm} when $\mean{\mu}\geq1$.

\subsection{Estimation of $\mu$}
We aim to prove that $\hatmu$ is always convergent in these cases.

\begin{prop}
	\label{prop:CVmuhat}
	If $\mean{\mu}\geq1$, then the estimators $\hatmu$ and $\hatmu^{\ast}$ respectively defined in eqs.\,\eqref{eq:estmu1} and \eqref{eq:estmuf}
	satisfy
	\[
	\hatmu\xrightarrow[h\to\infty]{}\mu
	\]
	and
	\[
	\hatmu^{\ast}\xrightarrow[h\to\infty]{}\mu
	\]
	almost surely.
	In addition, we have, for any $\ee>0$, that
	\[
	\sum_{h\geq 1}\mathbb{P}\left(\|\hatmu-\mu\|_{1}>\ee\right)<\infty.
	\]
\end{prop}

\noindent
Note that the result of Proposition \ref{prop:CVmuhat} is rather intuitive. Indeed, when the normal Galton-Watson subtrees are supercritical, the sample used in eq.\,\eqref{eq:estmu1} or in eq.\,\eqref{eq:estmuf} is a perturbation of size $h$ of a $\mu$ i.i.d.\ sample whose size is of order $m(\mu)^{h}$. Therefore, our primary concern is ensuring that this perturbation is not sufficiently large to hinder the estimation process.

\begin{proof}
	Recall that the spinal-structured tree $T$ can be decomposed as the grafting of a sequence $(G_{i,j})_{i,j\geq 1}$ of i.i.d.\ Galton-Watson trees with common birth distribution $\mu$ on the spine. For each of these trees, let us write $\overline{X}_{i,j,h}$, for $i,j,h\in\mathbb{N}$, the random vector defined by
	\[
	\overline{X}_{i,j,h}(k)=\sum_{\{v\in G_{i,j}\,:\,\mathcal{D}(v)< h\}}\mathbbm{1}_{\child{v}=k},\quad 0\leq k \leq N.
	\]
Let us highlight that $i$ corresponds to generations in the spinal-structured tree whereas $j$ corresponds to indices of the offsprings in a given generation. In addition, it is known (see for example \cite{devroye}) that the law of $\overline{X}_{i,j,h}$ conditional on $\card{\{v\in G_{i,j}\,:\,\mathcal{D}(v)< h\}}$ is multinomial with parameters $\dist$ and $\card{\{v\in G_{i,j}\,:\,\mathcal{D}(v)< h\}}$. From this, and from the independence of the $G_{i,j}$, it follows that the random variable $\overline{X}_h$ defined by
	\[
	\overline{X}_{h}=\sum_{i=1}^{h}\sum_{j=1}^{\S{i}-1}\overline{X}_{i,j,h-i}
	\]
	is, conditionally on $\card{T_{h}}$, a multinomial random variable with parameters $\card{T_{h}}-h$ and $\dist$ independent of $\spine_{h}$. Now, denoting by $\overline{\spine_{h}}$ the empirical distribution associated with $\spine_{h}$, that is
	\[
	\overline{\spine_{h}}(k)=\frac{1}{h}\sum_{i=1}^{h}\mathbbm{1}_{S_{i}=k},\quad \forall\,k\in\{1,\ldots,\pmax\},
	\]
	where $S_{1},\ldots, S_{h}$ denote the numbers of special children of the individuals of the spine,
	it is easily seen that
	\[
	\hatmu=\left(1-\frac{h}{\card{T}_{h}}\right)\overline{X}_{h}+\frac{h}{\card{T}_{h}}\overline{\spine_{h}}.
	\]
	Now, taking $\ee>0$, Pinsker's inequality entails that
	\[
	\mathbb{P}\left(\|\hatmu-\mu\|_{1}>\sqrt{\ee/2}\right)\leq\mathbb{P}\left(\KL{\hatmu}{\mu}>\ee\right).
	\]
	The convexity of the Kullback-Leibler divergence gives, with $\alpha_{h}:=\cfrac{h}{\card{T}_{h}}$, that
	\begin{multline*}
		\mathbb{P}\left(\|\hatmu-\mu\|_{1}>\sqrt{\ee/2}\right)\leq	\mathbb{P}\left((1-\alpha_{h})\KL{\overline{X}_{h}}{\mu}+\alpha_{h}\KL{\overline{\spine_{h}}}{\mu}>\ee\right)\\\leq \mathbb{P}\left((1-\alpha_{h})\KL{\overline{X}_{h}}{\mu}+\alpha_{h}\KL{\overline{\spine_{h}}}{\mu}>\ee,\ \alpha_{h}\KL{\nu}{\mu}<\frac{\ee}{2}\right)+\mathbb{P}\left( \alpha_{h}\KL{\nu}{\mu}\geq \frac{\ee}{2}\right).
	\end{multline*}
	Next, using the method of Lemma \ref{lem:pertMultEstimate}, one can show that for any $\delta>0$ there is a constant $C>0$ such that
	\[
	\mathbb{P}\left((1-\alpha_{h})\KL{\overline{X}_{h}}{\mu}+\alpha_{h}\KL{\overline{\spine_{h}}}{\mu}>\ee,\ \alpha_{h}\KL{\mu}{\nu}<\frac{\ee}{2}\Bigg| \card{T}_{h}\right)\leq C \exp\left(-\card{T_{h}}(\lopt\left(\alpha_{h}\right)-\delta)\right),
	\]
	with
	\begin{eqnarray*}
		&&\lopt(\alpha)\nonumber\\
		&:=&\!\!\!\!\!\!\!\!\inf_{(x,y)\in \measure^{2}}\!\left\{(1-\alpha)\KL{x}{\mu}+(1-\alpha)\KL{y}{\nu}\Bigg|(1-\alpha)\KL{x}{\mu}+\alpha\KL{y}{\mu}>\ee,\alpha \KL{\nu}{\mu}<\frac{\ee}{2}\right\} \nonumber \\
		&\geq&\!\!\!\!\!\!\!\!\inf_{\substack{(x,y)\in \measure^{2}\\ \alpha\in[0,1]}}\!\left\{(1-\alpha)\KL{x}{\mu}+(1-\alpha)\KL{y}{\nu}\Bigg|(1-\alpha)\KL{x}{\mu}+\alpha\KL{y}{\mu}\geq \ee,\alpha \KL{\nu}{\mu}\leq\frac{\ee}{2} \right\}.
		\label{eq:newOptim}
	\end{eqnarray*}
	As the feasible set of the r.h.s. is obviously compact, there exists a feasible point $(x^{\ast},y^{\ast},\alpha^{\ast})$ such that
	\begin{eqnarray*}
		&&(1-\alpha^{\ast})\KL{x^{\ast}}{\mu}+(1-\alpha^{\ast})\KL{y^{\ast}}{\nu}\\
		&=&\!\!\!\!\!\!\!\!\inf_{\substack{(x,y)\in \measure^{2}\\\alpha\in[0,1]}}\!\left\{(1-\alpha)\KL{x}{\mu}+(1-\alpha)\KL{y}{\nu}\Bigg|(1-\alpha)\KL{x}{\mu}+\alpha\KL{y}{\mu}\geq \ee,\alpha \KL{\nu}{\mu}\leq\frac{\ee}{2} \right\}.
	\end{eqnarray*}
	Assume that $(1-\alpha^{\ast})\KL{x^{\ast}}{\mu}+(1-\alpha^{\ast})\KL{y^{\ast}}{\nu}=0$, which readily implies that $x^{\ast}=\mu$ and $y^{\ast}=\nu$, but it is easily seen that for any $\alpha\in[0,1]$, the point $(\mu,\nu,\alpha)$ is not feasible. Hence, there exists a constant $\widetilde{C}>0$ independent of $\alpha$ such that
	\[
	\lopt(\alpha)\geq \widetilde{C}.
	\]
	Choosing $\delta<\widetilde{C}$, we get
\begin{eqnarray*}
	\mathbb{P}\left((1-\alpha_{h})\KL{\overline{X}_{h}}{\mu}+\alpha_{h}\KL{\overline{\spine_{h}}}{\mu}>\ee,\ \alpha_{h}\KL{\mu}{\nu}<\frac{\ee}{2}\Big| \card{T}_{h}\right) &\leq& C \exp\left(-\card{T}_{h}(\widetilde{C}-\delta)\right)\\
	&\leq& e^{-h(\widetilde{C}-\delta)},
	\end{eqnarray*}
	since $\card{T}_{h}\geq h$. Then,
	\[
	\mathbb{P}\left(\|\hatmu-\mu\|_{1}>\sqrt{\ee/2}\right)\leq e^{-hC}+\mathbb{P}\left( \alpha_{h}\KL{\mu}{\nu}\geq \frac{\ee}{2}\right).
	\]
	To ensure that the r.h.s.\ of the previous inequality is summable for $h\geq 1$, it thus remains to check that
	\[
	\sum_{h\geq 1}\mathbb{P}\left( \alpha_{h}\KL{\mu}{\nu}\geq \frac{\ee}{2}\right)<\infty.
	\]
	
\noindent
From this point, we assume that the birth distribution is critical. The supercritical is considered below.
So, to treat this, we use that the spinal-structured tree (excluding the spine) can be interpreted as a Galton-Watson tree with immigration with birth distribution $\mu$ and immigration $\tilde{\nu}$ given by $\tilde{\nu}_{k}=\nu_{k+1}$ for $k\geq 0$. It is known (see \cite{pakes}) that the generating function of $\card{T}_{h}$ is given by
\begin{equation}
\label{eq:generating}
\mathbb{E}\left[x^{\card{T}_{h}}\right]=x^{h}\prod_{i=0}^{h-1}B(g_{i}(x)),
\end{equation}
where $B:[0,1]\mapsto \mathbb{R}$ is the generating function of the law  $\tilde{\nu}$ and $g_{i}$ is the generating function of the total progeny of a Galton-Watson tree with law $\mu$ up to generation $i$, that is
\[
g_{i}(x)=\mathbb{E}\left[x^{\sum_{j=0}^{i}Z_{j}}\right], \quad \forall\,x\in[0,1],
\]
where $\left(Z_{i}\right)_{i\geq 0}$ is a standard Galton-Watson process with birth distribution $\mu$.
Now, denote \[
v_{h}=\mathbb{E}\left[x^{\frac{\card{T}_{h}}{h}}\right], \quad \forall\,h\geq 1.
\]
We then have that (the regularity of $B$ and $g_{i}$ is easily checked)
\[
\log\left(v_{h}\right)=\log\left(B(g_{i}(\theta_{h}))\right)+\left(\theta_{h+1}-\theta_{h}\right)\cfrac{g'_{i}(\eta_{h})B'(g_{i}(\eta_{h}))}{B(g_{i}(\eta_{h}))}
\]
with
\begin{equation}\label{eq:con}\theta_{h}=\exp\left(\frac{\log(x)}{h}\right)\text{ and }\eta_{h}\in(\theta_{h},\theta_{h+1}),\quad\forall\,h\geq 1.
\end{equation}
Hence, eq.\,\eqref{eq:generating} entails that
$$
\log\left(\cfrac{v_{h+1}}{v_{h}}\right)=\log(x)+\left(\theta_{h+1}-\theta_{h}\right)\sum_{i=0}^{h-1}\cfrac{g'_{i}(\eta_{h})B'(g_{i}(\eta_{h}))}{B(g_{i}(\eta_{h}))}+\log\left(B(g_{h}(\theta_{h+1}))\right).
$$
Now, as the sequence $g_{i}$ is monotonically decreasing and converging to some proper generating function $g$ (only in the critical case, see again \cite{pakes}),we have that
\[
\cfrac{g'_{i}(\eta_{h} )B'(g_{i}(\eta_{h}))}{B(g_{i}(\eta_{h}))}\leq  \cfrac{\mean{\tilde{\nu}}g_{i}'(\eta_{h})}{B(g(\eta_{h}))}.
\]
Now, as for $x\in(0,1)$,
\[
g_{i}'(x)=\mathbb{E}\left[\left(\sum_{j=0}^{i}Z_{j}\right)x^{\sum_{j=0}^{i}Z_{j}-1}\right],
\]
we have
\[
g_{i}'(x)\leq- \frac{e^{-1}}{x\log(x)}.
\]
It follows that
\begin{eqnarray*}
\limsup_{h\to\infty}\log\left(\cfrac{v_{h+1}}{v_{h}}\right)&\leq&\log(x)-\limsup_{h\to\infty}\left(\left(\theta_{h+1}-\theta_{h}\right)\frac{\mean{\tilde{\nu}}}{B(g(\eta_{h}))}\frac{he^{-1}}{\eta_{h}\log(\eta_{h})}+\log\left(B(g_{h}(\theta_{h+1}))\right)\right)\\
&=&\log(x)+\mean{\tilde{\nu}}e^{-1},
\end{eqnarray*}
where we used \eqref{eq:con} to get that
\[
\lim\limits_{h\to\infty}h\eta_{h}\log(\eta_{h})=\log(x).
\]
Now, as $x$ is arbitrary, it can always be chosen such that $\log(x)+\mean{\tilde{\nu}}e^{-1}<0$, which, by the ratio test, implies that, for such $x$,
\[
\sum_{h\geq 1}\mathbb{E}\left[x^{\frac{\card{T}_{h}}{h}}\right]<\infty.
\]
Finally, we have
\[
\mathbb{E}\left[x^{\frac{\card{T}_{h}}{h}}\right]\geq \mathbb{E}\left[x^{\frac{\card{T_{h}}}{h}}\mathbbm{1}_{\frac{\card{T_{h}}}{h}\leq c_{\ee}}\right]\geq x^{c_{\ee}}\mathbb{P}\left(\frac{\card{T_{h}}}{h}\leq c_{\ee}\right),
\]
where
$$c_{\ee}=\frac{2\KL{\mu}{\nu}}{\ee}.$$
From this, it follows that
\begin{equation}
\label{eq:summability}
\sum_{h\geq 1}\mathbb{P}\left(\frac{\card{T_{h}}}{h}\leq \frac{2\KL{\mu}{\nu}}{\ee}\right)<\infty.
\end{equation}
We now consider the case where $T_{h}$ is supercritical. A possible approach is to consider a coupling between the supercritical tree $\tree_{h}$ and a critical tree $\widetilde{\tree}_{h}$ using a thinning procedure, in order to get the estimate
\begin{equation}
\label{eq:revision}
\mathbb{E}\left[x^{\frac{\card{T}_{h}}{h}}\right]\leq \mathbb{E}\left[x^{\frac{\card{\widetilde{T}}_{h}}{h}}\right].
\end{equation}
Indeed, assume now $T_{h}$ to be supercritical. We consider a thinning of $T_{h}$ where each normal individual (and its decent) is killed independently with probability $p$. This induces a new tree $\widetilde{T}_{h}$ with new normal birth distribution $\widetilde{\mu}$ such that $\mean{\widetilde{\mu}}=p\mean{\mu}$. So taking $p=\mean{\mu}^{-1}$ implies that $\widetilde{T}_{h}$ is a spinal-structured tree with critical birth distribution. Hence, from the first part of the proof, we have that
\[
\sum_{h\geq 1}\mathbb{E}\left[x^{\frac{\card{\widetilde{T}}_{h}}{h}}\right]<\infty,
\]
for $x$ such that $\log(x)+\mean{\tilde{\nu}}e^{-1}<0$. But the thinning procedure used for constructing $\widetilde{T}_{h}$ directly implies that $\card{\widetilde{T}_{h}}\leq \card{T_{h}}$ almost surely, which gives \eqref{eq:revision}. This implies that 
\[
\sum_{h\geq 1}\mathbb{E}\left[x^{\frac{\card{T}_{h}}{h}}\right]<\infty.
\]
The remaining of the proof is the same as for the critical case. This ends the proof of the almost sure convergence of $\hatmu$. Concerning the almost sure convergence of $\hatmu^{\ast}$, first note that \eqref{eq:summability} implies that
\begin{equation}
\label{eq:CValphah}
\frac{h}{\card{T}_{h}}\xrightarrow[h\to\infty]{}0
\end{equation}
almost surely. Now, take any $\mathbf{s}\in\spineset{h}$, where we recall that $\spineset{h}$ is the set of spine candidates defined in Section \ref{s:algo:spine}, and consider the estimator $\mu^{\mathbf{s}}_{h}$ given by
\[
\mu^{\mathbf{s}}_{h}(i)=\frac{1}{\card{T}_{h}-h}\sum_{v\in\tree_h\setminus\mathbf{s}}\mathbbm{1}_{\child{v}=i},\quad \forall\,0\leq i \leq \pmax.
\]
 Thus,
\[
\left|\hatmu(i)-\hatmu^{\mathbf{s}}(i)\right|\leq \hatmu\left|1-\frac{\card{T}_{h}}{\card{T}_{h}-h}\right|+\frac{\card{\mathbf{s}}}{\card{T}_{h}-h}=\hatmu\left|1-\frac{\card{T}_{h}}{\card{T}_{h}-h}\right|+\frac{h}{\card{T}_{h}-h},
\]
and the result follows from \eqref{eq:CValphah} and the almost sure convergence of $\hatmu$.
\end{proof}

\subsection{Spine recovery}

Now, to go further in the proof of Theorem \ref{thm:mainThm}, we need to understand if we can recover enough of the spine in order to estimate $f$. To do so, the idea is to show that the Ugly Duckling $\hatspine$ contains a proportion of order $h$ of special individuals. Before this result, we need some preliminary lemmas. The first one concerns Kullback-Leibler divergence.
\begin{lem} \label{lem:LipKL}
Let $p\in\measure$ such that $\m{p}:=\inf_{0\leq i \leq \pmax}p_{i}>0$ and $\mean{p}\geq 1$. Then, there exists $\ee_{1}>0$ such that, for any $q,\hat{p}\in\measure$, we have
	\[
	\|p-\hat{p}\|_{1}<\ee_{1}\\
	\Longrightarrow \left|\KL{q}{\B{p}}-\KL{q}{\B{\hat{p}}}\right|\leq C_{1}\|p-\hat{p}\|_{1},
	\]
	where $C_{1}$ only depends on $p$.
\end{lem}
\begin{proof}
The proof has been deferred into Appendix~\ref{app:proof:LipKL}.
\end{proof}

\noindent
The following lemma concerns large deviations on the probability to distinguish two samples.
\begin{lem} \label{lem:pertMultEstimate}
Let $p$ and $q$ in $\measure$. Let $R$, $M$ and $S$ be three independent multinomial random variables with respective parameters $(h-n,p)$, $(h-n,q)$ and $(n,q)$, for some integers $h$ and $n$ such that $h>n$. Then, for any $\delta>0$ there exists a constant $C>0$ such that
	\[
	\mathbb{P}\left(\KL{M+S}{p}+\ee>\KL{R+S}{p}\right)\leq C \exp\left\{h(-(1-\alpha)\AH{p}{q}+\ee-\delta) \right\},
	\]
	where $\mathfrak{D}$ is the divergence defined in eq.\,\eqref{eq:AHDivergence}. In addition the constant $C$ only depends on $\pmax$ and $\delta$.
\end{lem}
\begin{proof}
The proof has been deferred into Appendix~\ref{app:proof:pertMultEstimate}.
\end{proof}

\noindent
We can finally come to the proof of Theorem \ref{thm:mainThm}.
\begin{proof}[Proof of Theorem \ref{thm:mainThm}, critical and supercritical cases]
Let us recall that, for any element $\mathbf{s}\in\spineset{h}$, $\bar{\mathbf{s}}$ is defined as the random vector given by
	\[
	\bar{\mathbf{s}}_{i}=\frac{1}{h}\sum_{v\in\mathbf{s}}\mathbbm{1}_{\child{v}=i},\quad \forall\,0\leq i \leq \pmax,
	\]
that is the empirical distribution of the numbers of children along $\mathbf{s}$. In addition, for any non-negative integer $l\leq h$, we denote by $\spineset{h}^{l}$ the subset of $\spineset{h}$ such that
	\[
	\spineset{h}^{l}=\left\{\mathbf{s}\in\spineset{h} \mid \card{(\mathbf{s}\cap \spine)}\leq l \right\}.
	\]
From the definition of $\estS{h}$, we have that, for any non-negative integer $l$,
	\begin{eqnarray*}
		\mathbb{P}\left(\#\left(\widehat{\mathcal{S}}_{h}\cap\mathcal{S} \right)\leq l \right)&=&\mathbb{P}\left(\max_{\mathbf{s}\in\spineset{h}^{l}}\KL{\bar{\mathbf{s}}}{\B{\hatmu}}>\max_{\mathbf{s}\in\spineset{h}\setminus\spineset{h}^{k}}\KL{\bar{\mathbf{s}}}{\B{\hatmu}}\right)\\
		&\leq& \mathbb{P}\left(\max_{\mathbf{s}\in\spineset{h}^{l}}\KL{\bar{\mathbf{s}}}{\B{\hatmu}}>\KL{\bar{\spine}}{\B{\hatmu}}\right)\\
		&=&\mathbb{P}\left(\bigcup_{\mathbf{s}\in\spineset{h}^{l}}\left\{\KL{\bar{\mathbf{s}}}{\B{\hatmu}}>\KL{\bar{\spine}}{\B{\hatmu}}\right\}\right).\\
	\end{eqnarray*}
Now, let $\ee>0$ such that $\ee<\ee_1$ where $\ee_1$ is defined in Lemma \ref{lem:LipKL}. Thus, we have according to Lemma \ref{lem:LipKL} that
	\[
	\mathbb{P}\left(\bigcup_{\mathbf{s}\in \spineset{h}^{l}}\left\{\KL{\bar{\mathbf{s}}}{\B{\hatmu}}>\KL{\bar{\spine}}{\B{\hatmu}}\right\},\ \|\hatmu-\mu\|_{1}\leq C_{1}^{-1}\ee\right)\leq \mathbb{P}\left(\bigcup_{\mathbf{s}\in \spineset{h}^{l}}\left\{\KL{\bar{\mathbf{s}}}{\B{p}}+\ee>\KL{\bar{\spine}}{\B{p}}\right\}\right),
	\]
where $C_{1}$ is also defined in Lemma \ref{lem:LipKL}.
Hence,  following the notation of eq.\,\eqref{eq:manytoone}, we have that
\begin{eqnarray*}
			&&\mathbb{P}\left(\#\left(\widehat{\mathcal{S}}_{h}\cap\mathcal{S} \right)\leq l \right) \nonumber \\
			&\leq&\mathbb{E}\left[\sum_{i=0}^{l}\sum_{j=1}^{S_i}\sum_{\{u\in G_{i,j}\,:\,\mathcal{D}(u)=i\}}\mathbbm{1}_{\KL{(1-\frac{i}{h})\overline{\mathcal{A}(u)}+\frac{i}{h}\overline{\spine_{i}}}{\B p}+\ee>\KL{\overline{\spine_{h}}}{\B p}}\right]+\mathbb{P}\left(\|\hatmu-\mu\|_{1}> C_{1}^{-1}\ee\right)\\
			&=&\mathbb{E}\left[S\right]\sum_{i=0}^{l}\mathbb{E}\left[\sum_{\{u\in G\,:\,\mathcal{D}(u)=i\}}\mathbbm{1}_{\KL{\frac{i}{h}\overline{\mathcal{A}(u)}+(1-\frac{i}{h})\overline{\spine_{i}}}{\B p}+\ee>\KL{\overline{\spine_{h}}}{\B p}}\right]+\mathbb{P}\left(\|\hatmu-\mu\|_{1}> C_{1}^{-1}\ee\right),
\end{eqnarray*}
where $G$ is some Galton-Watson tree with birth distribution $\mu$ and we recall that $\spine_{i}=(S_{1},\ldots,S_{i})$ are the $i$ first elements of the spine.
Now, applying the many-to-one formula \eqref{eq:manytoone}, we get
	\begin{eqnarray*}
		&&\mathbb{E}\left[\sum_{\{u\in G\,:\,\mathcal{D}(u)=h\}}\mathbbm{1}_{\KL{\frac{i}{h}\overline{\mathcal{A}(u)}+(1-\frac{i}{h})\overline{\spine_{h-i}}}{p}+\ee>\KL{\bar{\spine}}{p}}\right]\\
		&=&  m^{i}\mathbb{P}\left(\KL{\frac{i}{h}\overline{X}+\left(1-\frac{i}{h}\right)\overline{\spine_{h-i}}}{p}+\ee>\KL{\overline{\spine_{h}}}{p}\right),
	\end{eqnarray*}
where $\overline{X}$ is the empirical distribution of an i.i.d. sample $X_1,\ldots,X_i$ with law given by $\B{\mu}$ independent of $\spine$.

Now, let $\rho>0$ be such that $\log\left(\mean{\mu}\right)-\AH{\B{\mu}}{\nu}+\rho<0$. Thus, according to Lemma \ref{lem:pertMultEstimate} and Theorem \ref{thm:optim}, we can choose $\delta>0$ and $\ee$ small enough such that $V(\alpha,\ee)\geq v(\alpha)-\rho$ and
	\[
	\mathbb{P}\left(\#\left(\widehat{\mathcal{S}}_{h}\cap \mathcal{S} \right)\leq l \right)\leq C_{\delta} \mathbb{E}[S]\sum_{i=0}^{l}m^{h-i}\exp\left(-h \left(v\left(\frac{i}{h}\right)-\rho\right)+h\delta\right)+\mathbb{P}\left(\|\hatmu-\mu\|_{1}> C_{1}^{-1}\ee\right),
	\]
for some constant $C_{\delta}$ provided by Lemma \ref{lem:pertMultEstimate}.
Now, let $\eta>0$, we have, setting $\mathcal{E}_{h}=h-\#\left(\widehat{\mathcal{S}}_{h}\cap \mathcal{S} \right)$,
	\begin{eqnarray*}
		\mathbb{P}\left(\frac{\mathcal{E}_{h}}{h}>\eta\right)&=&\mathbb{P}\left(\#\left(\widehat{\mathcal{S}}_{h}\cap \mathcal{S}\right)\leq \lfloor(1-\delta) h\rfloor\right)
		\\&\leq& C_{\delta}\mathbb{E}[S]\sum_{\alpha\in L_h}\exp\left(h\left((1-\alpha)\log(m)- v\left(\alpha\right)+\rho+\delta\right)\right)+\mathbb{P}\left(\|\hatmu-\mu\|_{1}> \ee\right),
	\end{eqnarray*}
with $L_h=\left\{\frac{i}{h}\mid0\leq i\leq \lfloor (1-\delta)h\rfloor\right\}.$ One can now easily control the probability by
	\begin{equation*}
		\mathbb{P}\left(\frac{\mathcal{E}_{h}}{h}>\eta\right)\leq C_{\delta}\mathbb{E}[S](h+1)\exp\left(h\sup_{\alpha\in [0,(1-\eta)]}\Big\{(1-\alpha)\log(m)- V\left(\alpha\right)+\rho+\delta\Big\}\right).
	\end{equation*}
Now, let us denote
	\[
	\alpha_\eta:=\argmax_{\alpha \in [0,(1-\eta)]}(1-\alpha)\log(m)- V\left(\alpha\right).
	\]
Since $(1-\alpha)\log(m)- V\left(\alpha\right)<0$ for all $\alpha\in [0,1)$, we have that
	\[
	\alpha_\eta\xrightarrow[\eta\to 0]{}1,
	\]
in virtue of the continuity of $V$. Now, according to Theorem \ref{thm:optim}, we have that
	\[
	(1-\alpha)\log(m)- v\left(\alpha\right)\sim_{\alpha \to 1}(1-\alpha)\left(\log(m)-\bah{p}{q}\right).
	\]
Thus, for a fixed $\kappa>0$ and for $\eta$ small enough, we have
	\[
	(1-\alpha_\eta)\log(m)-v(\alpha_\eta)\leq (1-\kappa)(1-\alpha_\eta)(\log(m)-\bah{p}{q}),
	\]
which gives
	\begin{equation*}
		\mathbb{P}\left(\frac{\mathcal{E}_{h}}{h}>\eta\right)\leq C_{\delta}\mathbb{E}[S](h+1)\exp\left(h\big\{(1+\kappa)(1-\alpha_\eta)(\log(m)-\bah{p}{q})+\rho+\delta\big\}\right).
	\end{equation*}
Thus, for $\rho$ and $\delta$ small enough, we have that $\mathbb{P}\left(\frac{\mathcal{E}}{h}>\eta\right)$ converges to zero exponentially fast, which entails that
	\[
	\frac{\mathcal{E}_{h}}{h}\xrightarrow[h\to\infty]{}0
	\]
almost surely and ends the proof.
\end{proof}


\section{Simulation study}
\label{s:num}

The numerical results presented in this section have been obtained using the \verb+Python+ library \verb+treex+ \cite{Azais2019} dedicated to tree simulation and analysis.

\subsection{Consistency of estimators}
\label{ss:illus}

This part is devoted to the illustration of the consistency result stated in Theorem~\ref{thm:mainThm} through numerical simulations. For each of 3 normal birth distributions $\mu$ (subcritical, critical, and supercritical), we have simulated $50$ spinal-structured trees until generation $h_{\max}+1$, with $h_{\max}=125$. The birth distribution of special nodes $\nu$ is obtained from $\mu$ and $f$ using \eqref{eq:biased:distrib}, where the only condition imposed to $f$ (in the critical and supercritical regimes) was that the convergence criterion $\mathcal{K}(\mu,\nu) = \log\,\mean{\mu}-\AH{\B{\mu}}{\nu}$ is negative. The values of the parameters selected for this simulation study are presented in Tab.~\ref{tab:parameters}.

\begin{table}[th]
\centering
\begin{tabular}{c||c|c|c||c|c|c||c|c|c}
& \multicolumn{3}{c||}{\textbf{Subcritical}}  &   \multicolumn{3}{c||}{\textbf{Critical}} &   \multicolumn{3}{c}{\textbf{Supercritical}} \\ \hline
$k$ & 0 & 1 & 2  & 0& 1 & 2 & 0 & 1 & 2 \\ \hline
$\mu(k)$ & 0.35 & 0.4 & 0.25 & 0.4 & 0.2 & 0.4 & 0.29 & 0.4 & 0.31 \\ \hline
$f(k)$ & 0 & 1 & 3 & 0 & 1 & 3 & 0 & 1 & 4 \\ \hline
$\mathcal{K}(\mu,\nu)$ & \multicolumn{3}{c||}{$-0.116$} &\multicolumn{3}{c||}{$-0.017$} & \multicolumn{3}{c}{$-0.006$} \\ \hline
\end{tabular}
\caption{Values of the parameters $\mu$ and $f$ selected for the simulation of the spinal-structured trees in the subcritical, critical, and supercritical regimes, as well as the associated convergence criterion $\mathcal{K}(\mu,\nu) = \log\,\mean{\mu}-\AH{\B{\mu}}{\nu}$.}
\label{tab:parameters}
\end{table}

For each of these trees, we have estimated the unknown model parameters for observation windows $h$ between $5$ and $h_{\max}$ with a step of $5$. The normal birth distribution is estimated twice: by the (biased) maximum likelihood estimator $\hatmu$ given in \eqref{eq:estmu1} and by the corrected estimator $\hatmu^\star$ defined in \eqref{eq:estmuf}. The transform function $f$ is estimated by $\hatf$ defined in \eqref{eq:estf}. Finally, the special birth distribution $\nu$ is estimated by,
$$\forall\,0\leq k\leq N,~\widehat{\nu}_h(k) \propto \hatf(k)\hatmu^\star(k).$$
For these 4 numerical parameters, we have computed the error in $L^1$-norm (since $f$ is identifiable only up to a multiplicative constant, both $f$ and $\hatf$ were normalized so that their sum is $1$ before). The spine is estimated by the Ugly Duckling $\hatspine$ defined in \eqref{eq:def:hatspine}. In this case, the estimation error is given by the proportion of special nodes not recovered by $\hatspine$.

\begin{figure}[th]
\centering
\begin{tabular}{cc}
\includegraphics[width=0.48\textwidth]{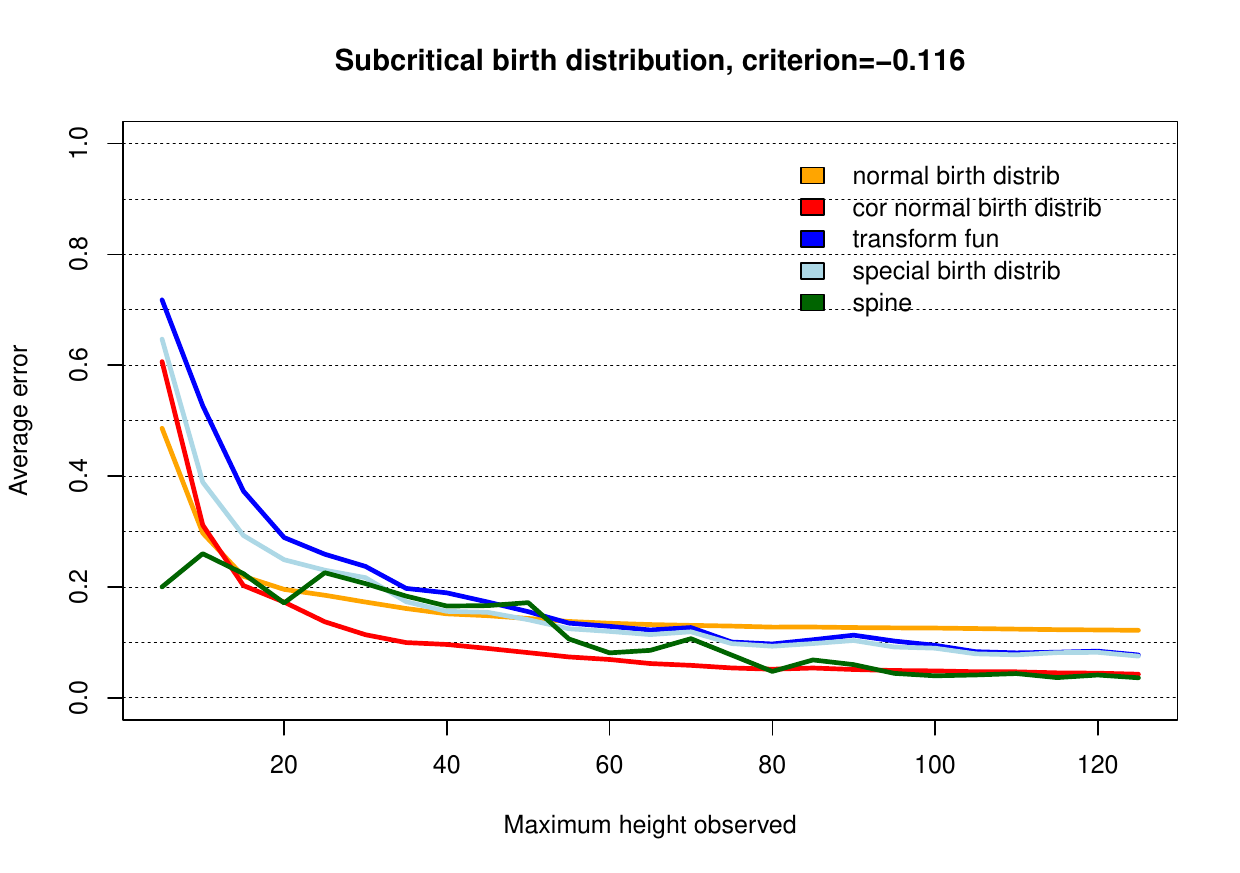}&
\includegraphics[width=0.48\textwidth]{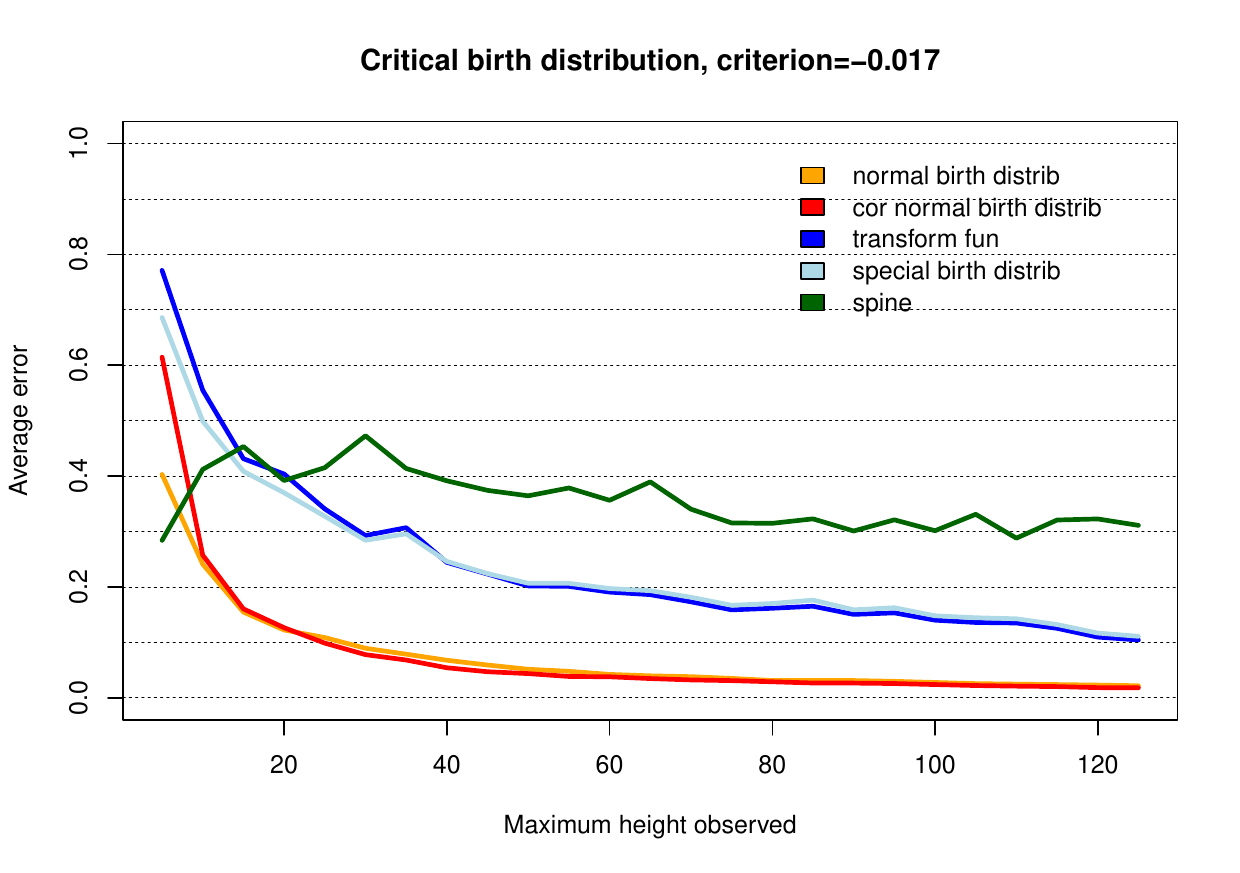}\\
\includegraphics[width=0.48\textwidth]{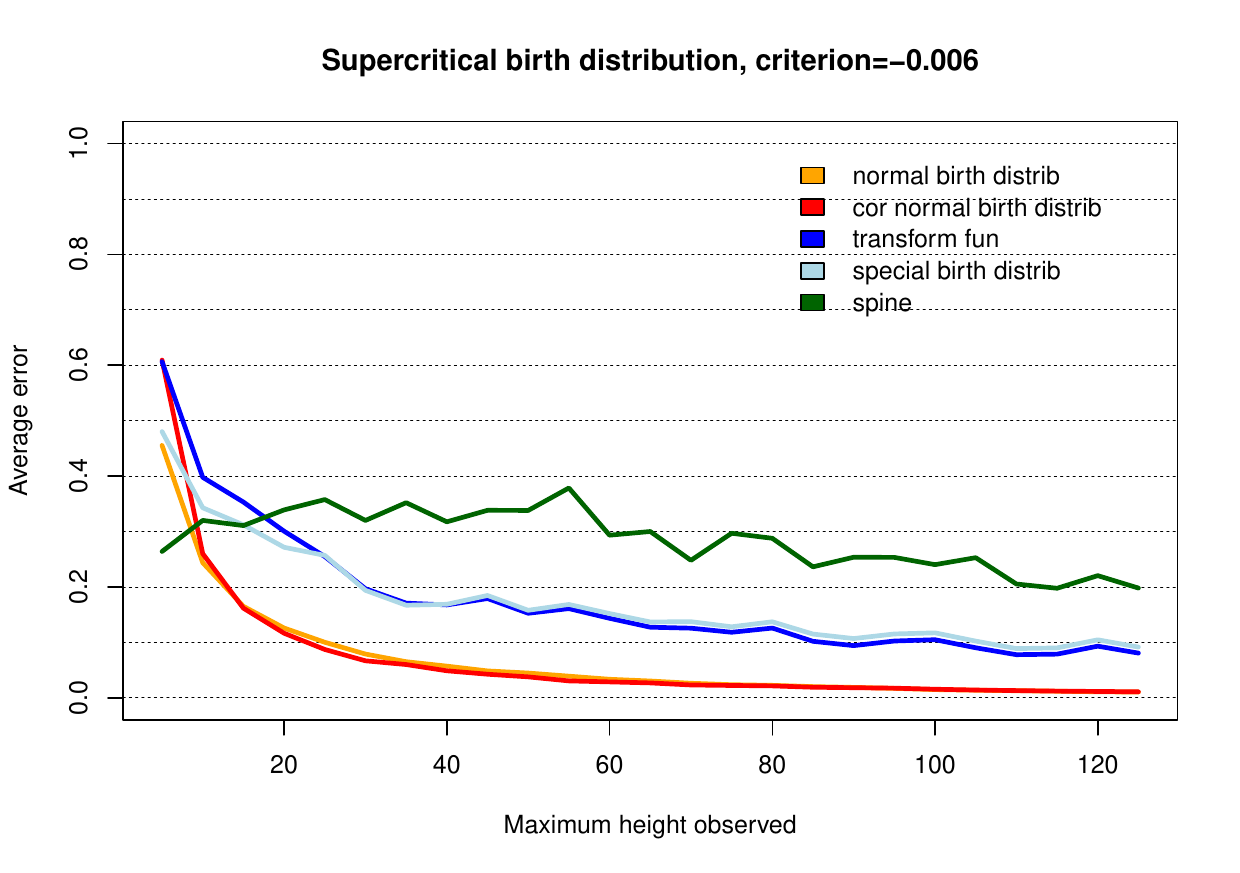}&
\end{tabular}
\caption{Average error as a function of the maximum height observed in the estimation of the unknown parameters (orange and red: $\mu$, blue: $f$, light blue: $\nu$, and green: $\mathcal{S}$) of a spinal-structured tree in the 3 growth regimes (top left: subcritical, top right: critical, and bottom left: supercritical). Parameter values can be read in Tab.~\ref{tab:parameters}.
}
\label{fig:num:res}
\end{figure}

The average errors computed for each of the $5$ estimators from varying observation windows $h$ are presented for the 3 growth regimes in Fig.~\ref{fig:num:res}. First of all, one can remark that the $5$ average errors tend to vanish when $h$ increases (even if it is with different shapes) whatever the growth regime (the convergence criterion is checked in the 3 examples). This illustrates the consistency of the estimators stated in Theorem~\ref{thm:mainThm}. However, additional pieces of information can be obtained from these simulations:
\begin{itemize}
\item It can be remarked that the correction of $\hatmu$ is useful only in the subcritical regime. In the two other regimes, one can indeed observe that the errors related to $\hatmu$ and $\hatmu^\star$ are almost superimposed. This is due to the fact that, in these growth regimes, the number of normal nodes is sufficiently large (compared to the number of special nodes) so that the bias of the maximum likelihood estimator vanishes.
\item The estimators of $f$ and $\nu$ are clearly less accurate than $\hatmu^\star$, in particular in the critical and supercritical regimes. A first but likely negligible reason is that $\hatf$ is computed from $\hatmu^\star$, which should only add an error to the one associated with the latter. Furthermore, the number of special nodes (used to estimate $\hatf$) is smaller than the number of normal nodes (used to estimate $\hatmu^\star$).
\item The estimator of the spine seems to converge, but slowly than the other estimates. However, we emphasize that, when $h$ increases, the number of unknown node types increases as well, contrary to $\mu$, $f$, and $\nu$, which dimension is fixed. It is thus expected to observe a slower convergence rate.
\end{itemize}


\subsection{Asymptotic test of conditioned Galton-Watson trees}
\label{ss:test}

When observing a population modeled by a Galton-Watson tree, it is of first importance to know whether it has been conditioned to survive or not, in particular when the birth distribution is subcritical. Here we show how the theoretical contributions of this paper can be used to develop an asymptotic test to answer this question.

We observe a subcritical tree $T$ until generation $h$ and would like to test the null hypothesis: $T$ is a Galton-Watson tree conditioned to survive until (at least) generation $h$. In the framework of spinal-structured trees and approximating conditioned Galton-Watson trees by Kesten's model \cite{kesten1986subdiffusive}, this is equivalent to test $f\propto\text{Id}$, which simplifies the construction of the test but also provides a further motivation for the class of models considered in this paper. As in Subsection~\ref{ss:illus}, we assume that $\sum_{i=0}^Nf(i)=1$.

\begin{figure}[t]
\centering
\begin{tabular}{cc}
\includegraphics[width=0.48\textwidth]{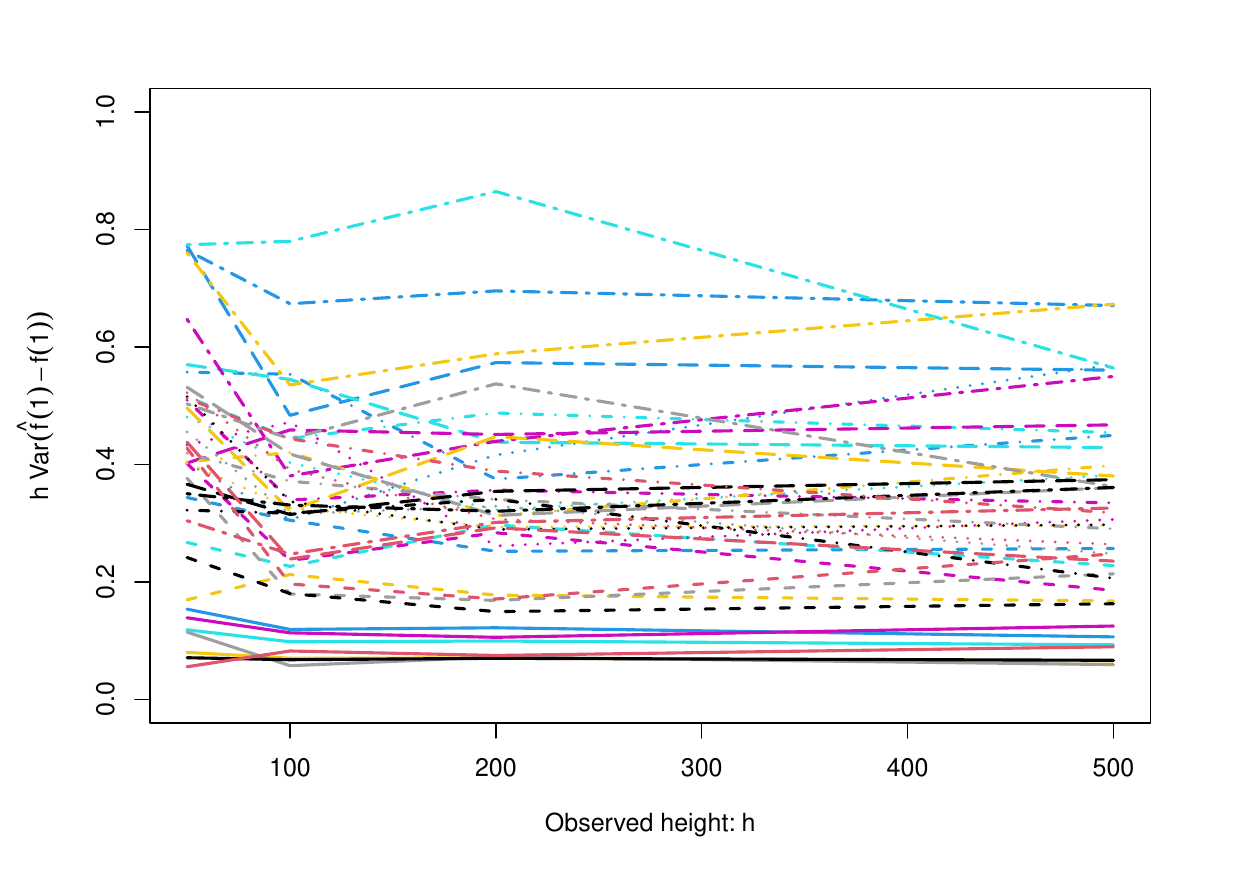} & \includegraphics[width=0.48\textwidth]{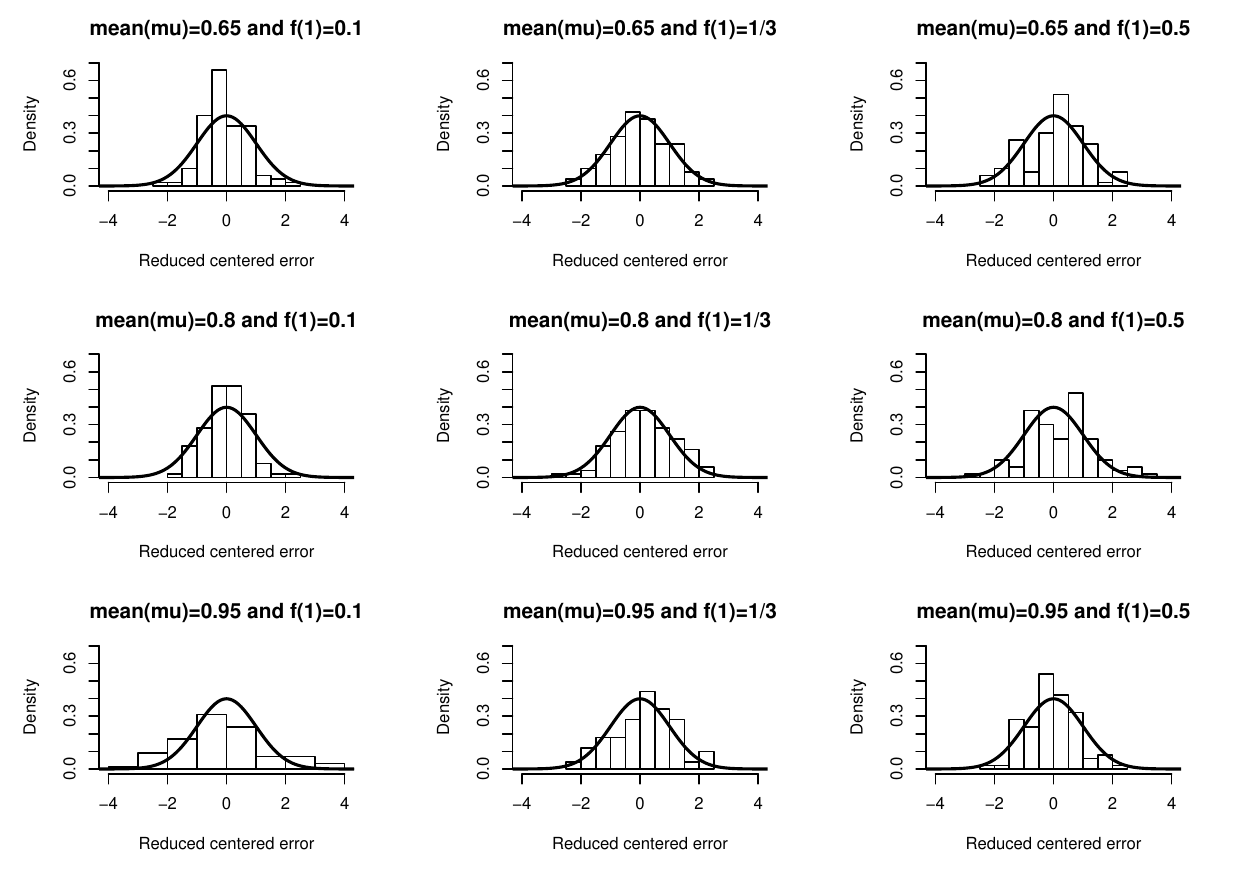}
\end{tabular}
\caption{Estimates of $h\,\mathbb{V}\text{ar}(\hatf(1)-f(1))$ from samples of $100$ spinal-structured trees simulated with various parameters $\mu$ and $f$ with $0.5<m(\mu)<1$ and $0<f(1)<0.5$ (left) and empirical distribution of the reduced centered error $\sqrt{h}(\hatf(1)-f(1))/\sigma(\hatmu,\hatf)$ for some of these parameters (right) with a comparison to the Gaussian distribution (thick black line).}
\label{fig:est:tcl}
\end{figure}

The construction of a test statistic for $f$ requires both a consistent estimator and some knowledge on its asymptotic behavior. The latter is sorely lacking but can be estimated from numerical simulations. In the sequel, we restrict ourselves to binary  ($f(1)$ is therefore sufficient to know $f$) spinal-structured trees with $0.5<m(\mu)<1$ and $0<f(1)<0.5$, i.e. $f$ is increasing because $f(1)+f(2)=1$. We suspect that $\hatf(1)$ satisfies a central limit theorem with rate $\sqrt{h}$. However we want to check if this rate seems to be adequate, then we need to estimate its asymptotic variance, possibly as a function of both $\mu$ and $f$. To this end, we have estimated $h\,\mathbb{V}\text{ar}(\hatf(1)-f(1))$ from simulated samples of spinal-structured trees from various values of $\mu$ and $f$ within the range specified above: the results are presented in Fig.~\ref{fig:est:tcl} (left). First, we observe that $h\,\mathbb{V}\text{ar}(\hatf(1)-f(1))$ seems to be constant in $h$ whatever the value of the two parameters, which validates the rate $\sqrt{h}$. In addition, the asymptotic variance clearly depends on the parameters, but can be accurately predicted from them by a linear regression,
$$ \sigma^2(\mu,f) = 0.4611141 - 0.5561625\times m(\mu) + 1.0688165\times f(1).$$
We display in Fig.~\ref{fig:est:tcl} (right) the distribution of $\sqrt{h}(\hatf(1)-f(1))/\sigma(\hatmu,\hatf)$, which is very close to the Gaussian distribution, as expected.

\begin{figure}[t]
\centering
\begin{tabular}{c|c}
Kesten's tree & Population model with competition \\
& \\
\includegraphics[width=0.48\textwidth]{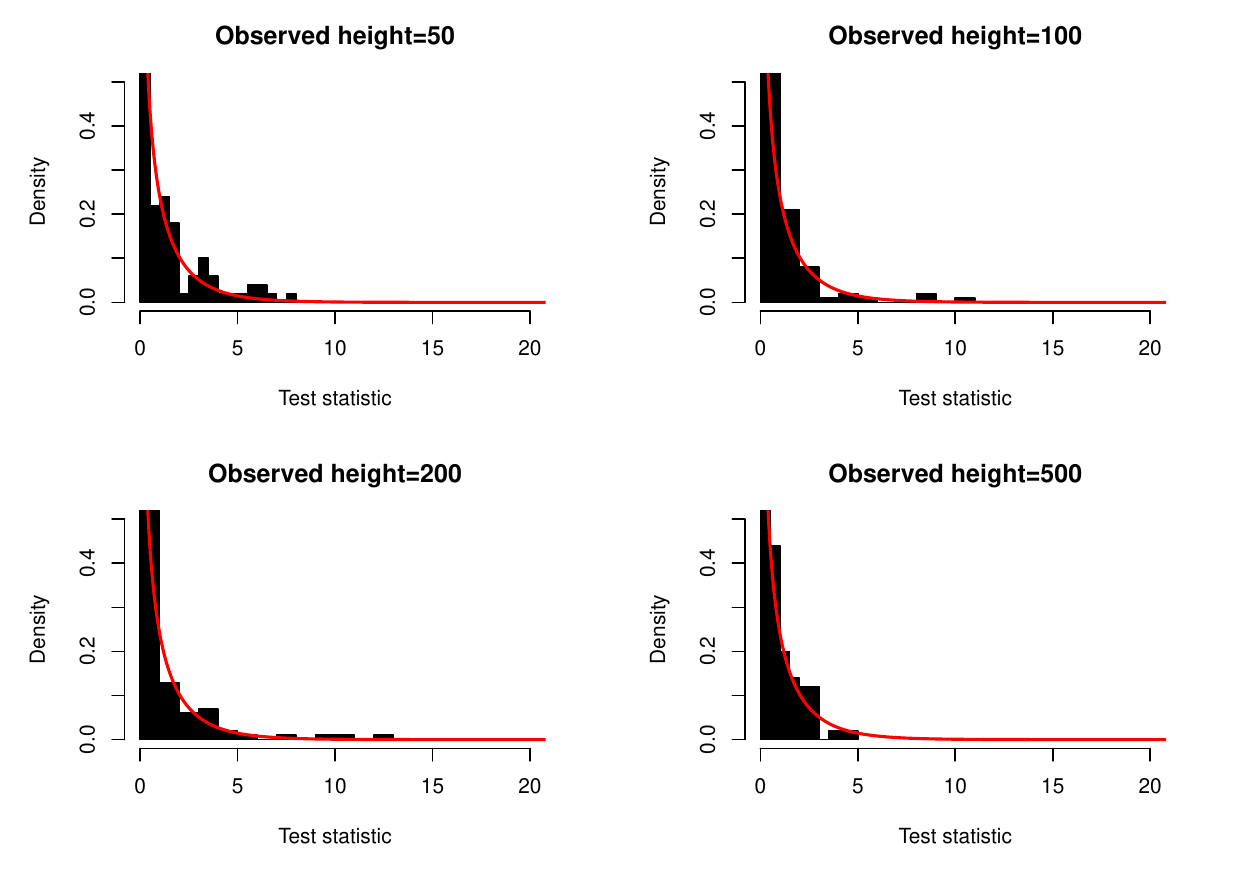} & \includegraphics[width=0.48\textwidth]{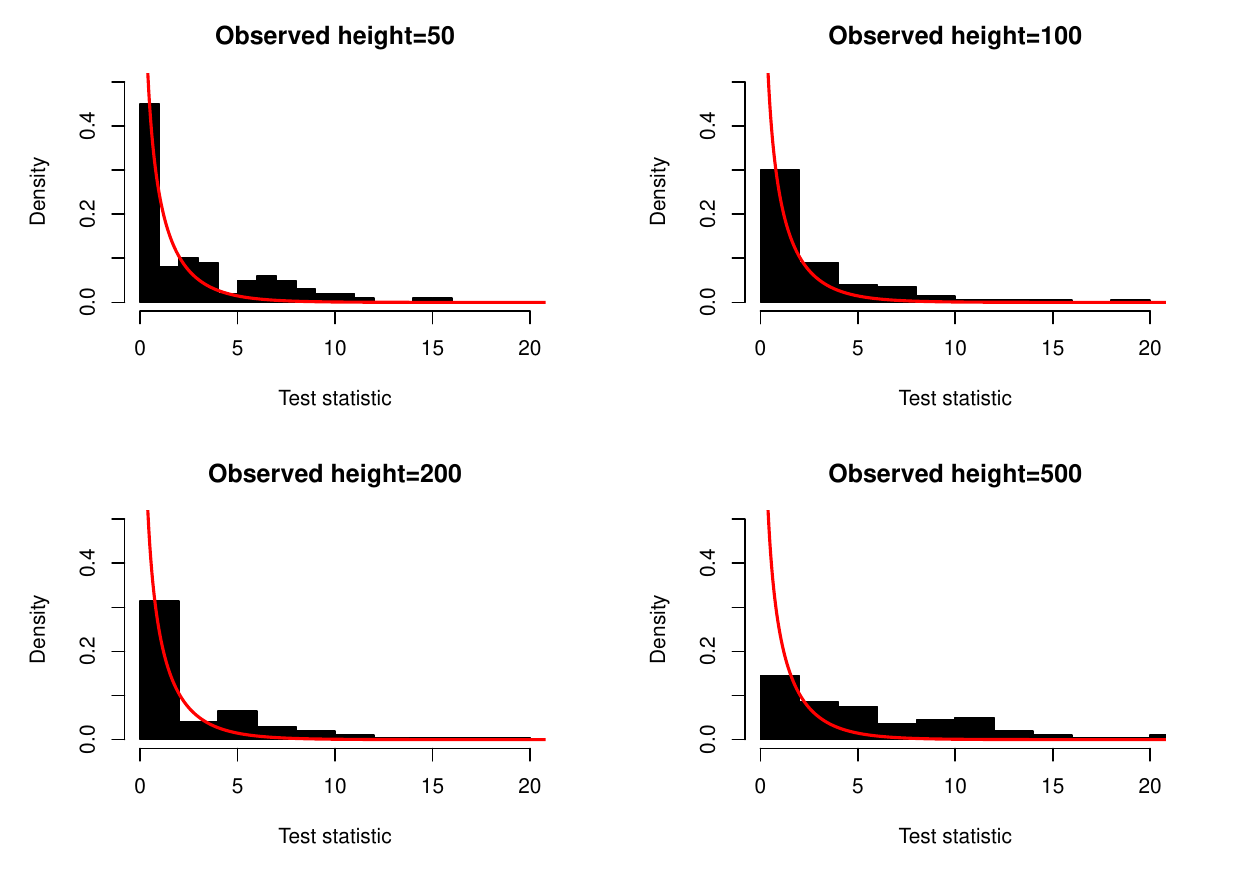}
\end{tabular}
\caption{Empirical distribution of the test statistic $Q_h$ obtained from samples of $100$ Kesten's trees with normal birth distribution $(0.55,0.2,0.25)$ (left) and from samples of $100$ Galton-Watson trees with competition (right), both with a comparison to the $\chi^2(1)$ distribution (red line).}
\label{fig:testH1}
\end{figure}

Relying on this short simulation study and recalling that $f(1)=1/3$ in Kesten's model ($f\propto\text{Id}$ and $f(1)+f(2)=1$), we introduce the test statistic
$$Q_h = \frac{h(\hatf(1)-1/3)^2}{\sigma^2(\hatmu,\hatf)},$$
which approximately follows a $\chi^2(1)$ distribution when the underlying tree is sampled according to Kesten's model. Denoting $\mathbf{q}(x) = \mathbb{P}(\chi^2(1)>x)$, one rejects the hypothesis $f\propto\text{Id}$ with confidence level $1-\alpha$ when $Q_h>\mathbf{q}(1-\alpha)$. Fig.~\ref{fig:testH1} (left) illustrates the behavior of $Q_h$ when the tested hypothesis is true. Furthermore, Tab.~\ref{tab:testH0} shows that the test rejects the null hypothesis with approximately the expected frequency of error $\alpha$.

To go further, the behavior under the alternative hypothesis needs to be investigated. That is why we propose to apply the test to a population that does not follow Kesten's model. For this purpose, we consider a Galton-Watson model with competition: for any $k\geq0$, each node of the $k^\text{th}$ generation gives birth, with distribution $\mu_s$ depending on its size $s$, to a random number of children,
$$\mu_{s} = \left( \frac{1}{4} \left(1-\frac{1}{s}\right), \frac{3}{4}\left(1-\frac{1}{s}\right) , \frac{1}{s}\right).$$
It should be noted that $m(\mu_{s}) = \frac{3}{4} +\frac{5}{4s}$, which makes the population growth supercritical when $s\leq 4$, critical when $s=5$, and subcritical when $s\geq6$. Oscillating between exponential growth and decay, the population is likely to avoid extinction, without fitting to the behavior of a Galton-Watson tree conditioned on surviving. Fig.~\ref{fig:testH1} (right) provides the distribution of the test statistic $Q_h$ under this model: it is significantly different from the $\chi^2(1)$ distribution expected under the null hypothesis, even from small populations, and thus differentiates from conditioned Galton-Watson trees.

\begin{table}[th]
\centering
\begin{tabular}{c||c||c|c|c|c}
$\alpha$ & $\textbf{q}(1-\alpha)$ &  \multicolumn{4}{c}{\textbf{Empirical rejection rate}} \\
& & $h=50$ & $h=100$ & $h=200$ & $h=500$ \\ \hline
$10\%$ & $2.71$ & 19\% & 8\%& 15\%& 5\%\\ 
$5\%$ & $3.84$ & 10\% & 6\%& 9\%& 2\%\\ 
$1\%$ & $6.63$ & 2\% & 3\%& 4\%& 0\%\\ \hline
\end{tabular}
\caption{Empirical rejection rate measured when testing the null hypothesis $f\propto\text{Id}$ from samples of $100$ Kesten's trees with normal birth distribution $(0.55,0.2,0.25)$ observed until generation $h$ for various levels of confidence $1-\alpha$.}
\label{tab:testH0}
\end{table}

\appendix

\section{Proofs of intermediate lemmas}
\label{app:lem}
\subsection{Proof of Lemma~\ref{lem:nozero}} \label{app:proof:nozero}

Let $(x,y,z)\in\feas(\alpha,\ee)$. The proof is based on the fact that if $x$, $y$ or $z$ has some null coordinate, we may always perturbed these points in order to decrease the objective function while still remaining in $\feas(\alpha,\ee)$. In a sake of simplicity, we assume along the proof that $x_{i}>0$, $z_{i}>0$, and $y_{i}>0$, as soon as $i\neq0$. The other cases can be treated similarly and are left to the reader. The harder case is when we have $x_{0}=z_{0}=0$ and $y_{0}>0$, but, to give an example, we first treat the case where $x_{0}=0$, $z_{0}>0$ and $y_{0}>0$. In the sequel of the proof, for any $X\in\mathbb{R}^{\pmax+1}$, we denote by $I(X)$ the set given by
	\[
	I(X)=\left\{i\in\llbracket0,\pmax\rrbracket \mid X_{i}=0\right\},
	\]
	and
	$\nabla^{+}\KL{x}{p}$ the vector given by
	\[
	\nabla^{+}\KL{x}{p}=
	\left\{
	\begin{array}{lll}
	1+\log(x_i/p_i)& \text{for all} & i\in I^{c}(x),\\
	0& \text{for all}&i\in I(x),
	\end{array}\right.
	\]
which is the vector of directional derivatives of $\KL{\cdot\,}{p}$ in the directions where they are well-defined.
	
	\medskip
	
	\noindent
	{\bf Case 1: $x_{0}=0$, $z_{0}>0$ and $y_{0}>0$}
	
	\medskip
	
\noindent
We first show that, whenever $h\in \mathbb{R}^{\pmax+1}$ satisfies $h_{0}>0$, we have
	\begin{equation}
		\label{eq:nozeroeq1}
		\limsup_{\delta \to 0}\frac{f_{\alpha}(x+\delta h,y,z)-f_{\alpha}(x,y,z)}{\delta}<0.
	\end{equation}
For any $i\in I(x)$ and any $\delta>0$,
	\[
	f_{\alpha}(x+\delta \mathbf{e}_{i},y,z)-f_{\alpha}(x,y,z)=(1-\alpha)\delta \log\left(\frac{\delta}{p_{i}} \right),
	\]
where $\mathbf{e}_{i}$ the $i$th vector of the canonical basis of $\mathbb{R}^{\pmax+1}$.
This easily entails that
	\begin{equation}
	\label{eq:nozeroeq1.5}
	\lim_{\delta \to 0}\frac{f_{\alpha}(x+\delta \mathbf{e}_{i},y,z)-f_{\alpha}(x,y,z)}{\delta}=-\infty.
	\end{equation}
Now, take $h\in \mathbb{R}^{\pmax+1}$ such that $h_{0}>0$. We have that $h=h^{I(x)}+h^{I^c(x)}$ where $h^{I(x)}$ is given by
	\[
	h^{I(x)}=\left\{
	\begin{array}{ll}
		h_{i} & \text{if}~i\in I(x),\\
		0 & \text{else,}
	\end{array}\right.
	\]
and $h^{I^c(x)}=h-h^{I(x)}$. Thus, because $f$ has well-defined directional derivatives in the direction of the positive coordinate, we have
	\begin{equation}
		\label{eq:nozeroeq2}
		f_{\alpha}(x+\delta h,y,z)=f_{\alpha}(x+\delta h^{I(x)},y,z)+\delta \nabla^{+}f_{\alpha}(x,y,z)\cdot h^{I^c(x)}+o(\delta).
	\end{equation} 
Thus, eq.\,\eqref{eq:nozeroeq2} implies
	\[
	\limsup_{\delta \to 0}\frac{f_{\alpha}(x+\delta h,y,z)-f_{\alpha}(x,y,z)}{\delta}= \limsup_{\delta \to 0}\frac{f_{\alpha}(x+\delta h^{I(x)},y,z)-f_{\alpha}(x,y,z)}{\delta}+\nabla^{+}f_{\alpha}(x,y,z)\cdot h^{I^c(x)},
	\]
and \eqref{eq:nozeroeq1} now follows from \eqref{eq:nozeroeq1.5}.
Now, let $i^{\ast}\in \llbracket 0,\pmax\rrbracket$ such that
	\[
	\log\left(\frac{(1-\alpha)x_{i^{\ast}}+\alpha z_{i^{\ast}}}{p_{i^{\ast}}} \right)=\min_{0\leq i\leq \pmax}\log\left(\frac{(1-\alpha)x_{i}+\alpha z_{i}}{p_{i}} \right).
	\]
Because $(1-\alpha)x+\alpha z\in\measure$ and $p\in\measure$, we must have
	\[
	\log\left(\frac{(1-\alpha)x_{i^{\ast}}+\alpha z_{i^{\ast}}}{p_{i^{\ast}}} \right)<0.
	\]
Thus, taking $h= e_{0}-e_{i^{\ast}}$, we get
	\[
H_{\alpha,\ee}(x+\delta h,y,z)-H_{\alpha,\ee}(x,y,z)=\delta(1-\alpha)\left(\log\left(\frac{\alpha z_{0}}{p_{0}} \right)-\log\left(\frac{(1-\alpha)x_{i^{\ast}}+\alpha z_{i^{\ast}}}{p_{i^{\ast}}} \right)\right)+o(\delta).
	\]
For $\delta$ small enough, we obtain
	\[
H_{\alpha,\ee}(x+\delta h,y,z)\geq H_{\alpha,\ee}(x,y,z),
	\]
which implies that $(x+\delta h,y,z)\in\feas(\alpha,\ee)$. In addition, eq.\,\eqref{eq:nozeroeq1} implies $f(x+\delta h,y,z)<f(x,y,z)$ for $\delta$ small enough. Thus $(x,y,z)$ can not be a solution of problem \eqref{eq:optimproblem}.
	
\medskip
	
\noindent
{\bf Case 2: $x_{0}=z_{0}=0$ and $y_{0}>0$}
	
\medskip

\noindent		
This particular case raises a new difficulty. Informally, in such situation a perturbation of type $(x+\delta h,y,z+\delta \tilde{h})$ gives
	\[
	H_{\alpha,\ee}(x+\delta h,y,z+\delta \tilde{h})=H_{\alpha,\ee}(x,y,z)+\delta\log(\delta)+o(\delta\log(\delta)).
	\]
It follows that $H$ is decreasing in any direction of type $\delta(h,0,\tilde{h})$ and $(x+\delta h,y,z+\delta \tilde{h})$ may not be in $\feas(\alpha,\ee)$.
To overcome this problem, we consider perturbed points of the form
	\[
	\begin{cases}
		x^{\delta}=x+\delta \mathbf{e}_{0}-\delta \mathbf{e}_{i}+\delta\log(\delta)r,\\
		z^{\delta}=z+\delta \mathbf{e}_{0}-\delta \mathbf{e}_{i},
	\end{cases}
	\]
where $r\in\mathbb{R}^{\pmax+1}$ satisfies
	\begin{equation}
		\label{eq:Rcond}
		\begin{cases}
			r_{0}=0,\\
			\sum_{i=0}^{\pmax}r_{i}=0.
		\end{cases}
	\end{equation}
Thus, for $\delta$ small enough, it is granted that $z^{\delta}\in\measure$ and $x^{\delta}\in\measure$. Because, $x_{i}>0$ and $z_{i}>0$ for all $i>0$, we have that
	\begin{equation*}
		\KL{x^{\delta}}{p}=\KL{x}{p}+\delta\log(\delta)+\delta\log(\delta)r\bigcdot\nabla^{+}\KL{x}{p}+o(\delta\log(\delta)),
	\end{equation*}
and
	\begin{equation*}
		\KL{z^{\delta}}{q}=\KL{z}{q}+\delta\log(\delta)+o(\delta\log(\delta)),
	\end{equation*}
which gives that
	\begin{equation}
		\label{eq:pertF}
		f_{\alpha}(x^{\delta},y,z^{\delta})=f(x,y,z)+\delta\log(\delta)\left(1+(1-\alpha)r\bigcdot\nabla^{+}\KL{x}{p}\right)+o(\delta\log(\delta)).
	\end{equation}
Similarly, we get that
	\begin{equation}
		\label{eq:pertG}
		H_{\alpha,\ee}(x^{\delta},y,z^{\delta})=H_{\alpha,\ee}(x,y,z)+\delta\log(\delta)\left(1+(1-\alpha)r\bigcdot\nabla^{+}\KL{(1-\alpha)x+\alpha z}{p}\right)+o(\delta\log(\delta)).
	\end{equation}
Our next step is to show that for some choice of $r$ and $\delta$ small enough, we have $f(x^{\delta},y,z^{\delta})\leq f(x,y,z)$ and $h(x^{\delta},y,z^{\delta})\geq h(x,y,z)$, which imply $(x^{\delta},y,z^{\delta})\in \feas(\alpha,\ee)$ and that $(x,y,z)$ is not a minimizer of $f$ among the feasible set $\feas(\alpha,\ee)$. To show this, in virtue of eqs.\,\eqref{eq:pertF} and \eqref{eq:pertG}, we only need to find some $r\in\mathbb{R}^{\pmax+1}$ satisfying conditions \eqref{eq:Rcond} and
	\begin{equation}
		\label{eq:Rcond2}
		\begin{cases}
			-r\bigcdot\nabla^{+}\KL{x}{p}\leq 1\\
			r\bigcdot\nabla^{+}\KL{(1-\alpha)x+\alpha z}{p}\leq -1,
		\end{cases}
	\end{equation}
in particular because $\delta\log(\delta)<0$ for $\delta$ small enough. According to Farkas' lemma, such $r$ exists as soon as there is no solution $(u,v,w)$ to the problem
	\begin{equation}
		\label{eq:fark}
		\begin{cases}
			-u\log\left(\cfrac{x_{i}}{p_{i}} \right)+v\log\left(\cfrac{(1-\alpha)x_{i}+\alpha z_{i}}{p_{i}} \right)+w=0,\quad \forall\,i>0,\\
			u-v<0,\\
			u>0,\ v>0,\ w>0.
		\end{cases}
	\end{equation}
Assume that $(u,v,w)$ is a solution to problem \eqref{eq:fark}. Thus, for all $i>0$, 
	\[
	x_{i}=e^{\frac{w}{u}}p_{i}\left(\cfrac{(1-\alpha)x_{i}+\alpha z_{i}}{p_{i}}\right)^{\frac{v}{u}}.
	\] 
Hence, according to Jensen's inequality and the conditions of problem \eqref{eq:fark},
	\[
	1=e^{\frac{w}{u}}\sum_{i=1}^{\pmax}p_{i}\left(\cfrac{(1-\alpha)x_{i}+\alpha z_{i}}{p_{i}}\right)^{\frac{v}{u}}\geq e^{\frac{w}{v}}>1,
	\]
which is absurd. Thus, problem \eqref{eq:fark} has no solution and Farkas' lemma entails that there exists $r\in \mathbb{R}^{\pmax+1}$ such that conditions \eqref{eq:Rcond} and \eqref{eq:Rcond2} are satisfied. 
	
The method is similar if we have more than one zero. This ends the proof.


\subsection{Proof of Lemma~\ref{lem:LipKL}} \label{app:proof:LipKL}
We begin the proof by showing that the Kullback-Leibler divergence $(q,p)\mapsto\KL{q}{p}$ is locally Lipschitz in the second variable away from $0$ and that this holds uniformly with respect to the first variable. We have
	\[
	\nabla_{2}\KL{q}{p}=\left(-\frac{q_{i}}{p_{i}} \right)_{1\leq i\leq \pmax},
	\]
where $\nabla_{2}$ denotes the gradient with respect to the second variable. Hence,
	\[
	\norm{\nabla_{2}\KL{q}{p}}_{1}\leq \frac{\pmax}{\m{p}}.
	\]
Given $0<\ee<\m{p}$, as soon as $\|p-\hat{p}\|_{1}<\ee$, we have
	\[
	\sup_{\{\hat{p}\in\measure\,|\,\|p-\hat{p}\|_{1}<\ee\}}\norm{\nabla_{2}\KL{q}{\hat{p}}}_{1}\leq \frac{N}{\m{p}-\ee},
	\] 
which entails that
	\begin{equation*} 
		\left|\KL{q}{p}-\KL{q}{\hat{p}}\right|\leq \frac{N}{\m{p}-\ee}\|p-\hat{p}\|_1,\quad\forall\,\hat{p}\in\measure~\text{s.t.}~\|p-\hat{p}\|_{1}<\ee.
	\end{equation*}
To go further, we need to investigate the effect of a perturbation of $p$ on $\B{p}$, where the operator $\B$ was defined in eq.\,\eqref{eq:biasOperator}.
Now, one can easily see that, on the open set $\{p\in\measure\mid\mean{p}>1/2\}$, we have
	\[
	\|\nabla\B{p}\|_{1}\leq 2\pmax.
	\]
As $\mean{p}\geq1$, there exists $\ee>0$, such that $\mean{\hat{p}}>1/2$ for all $\hat{p}\in\ball{p}{\ee}$. Thus, for $\hat{p}\in \measure\cap\ball{p}{\ee}$, we have
	\[
	\|\B{\hat{p}}-\B{p}\|_{1}\leq 2N\|p-\hat{p}\|_{1},
	\]
which ends the proof.


\subsection{Proof of Lemma~\ref{lem:pertMultEstimate}} \label{app:proof:pertMultEstimate}
First, we have
\begin{eqnarray}
&&\mathbb{P}\left(\KL{M+S}{p}>\KL{R+S}{p}\right) \nonumber \\
&=&\sum_{r_{1}+\ldots+r_{\pmax}=h-n}\sum_{m_{1}+\ldots+m_{\pmax}=h-n}\sum_{s_{1}+\ldots+s_{\pmax}=h}\!\!\!\left((h-n)!\right)^2h!\prod_{i=1}^{\pmax}\frac{p_{i}^{r_{i}}q_{i}^{m_{i}+s_{i}}}{r_{i}!m_{i}!s_{i}!}\mathbbm{1}_{\KL{\bar{r}}{p}+\ee>\KL{\bar{r}}{p}}. \label{eq:mult}
\end{eqnarray}
In addition, one can easily check that
\begin{eqnarray*}
&&\prod_{i=1}^{\pmax}p_{i}^{r_{i}}q_{i}^{m_{i}+s_{i}} \\
&=&\exp\left(-(n-h)\KL{\bar{r}}{p}-(n-h)\KL{\bar{m}}{p}-h\KL{\bar{s}}{q}\right)\prod_{i=1}^{\pmax}\left(\cfrac{r_{i}}{n-h}\right)^{r_{i}}\left(\cfrac{m_{i}}{n-h}\right)^{m_{i}}\left(\cfrac{s_{i}}{h}\right)^{s_{i}}.
\end{eqnarray*}
In addition, since
\[
(h-n)!(h-n)! h!\prod_{i=1}^{\pmax}\left(r_{i}!m_{i}!s_{i}!\right)^{-1}\prod_{i=1}^{\pmax}\left(\cfrac{r_{i}}{n-h}\right)^{r_{i}}\left(\cfrac{m_{i}}{n-h}\right)^{m_{i}}\left(\cfrac{s_{i}}{h}\right)^{s_{i}}\leq 1,
\]
we obtain from eq.\,\eqref{eq:mult} that
\begin{eqnarray*}
&&\mathbb{P}\left(\KL{M+S}{p}>\KL{R+S}{p}\right)\\
&\leq&\sum_{\begin{subarray} ~r_{1}+\ldots+r_{\pmax}=h-n\\m_{1}+\ldots+m_{\pmax}=h-n\\s_{1}+\ldots+s_{\pmax}=h\end{subarray}}\exp\Big(-(n-h)\KL{\bar{r}}{p}-(n-h)\KL{\bar{m}}{p}-h\KL{\bar{s}}{q}\Big)\mathbbm{1}_{\KL{\bar{r}}{p}+\ee>\KL{\bar{r}}{p}},
\end{eqnarray*}
which leads to
\begin{equation*}
	\mathbb{P}\left(\KL{M+S}{p}>\KL{R+S}{p}\right)\leq\dbinom{h-n+\pmax-1}{\pmax-1}^{2}\dbinom{h+\pmax-1}{\pmax-1}\exp\left(-hV(\alpha,\ee)\right),
\end{equation*}
where $V(\alpha,\ee)$ is defined in eq.\,\eqref{eq:valueF} 
and $\alpha=h/n$.
Now, as
\[
\dbinom{h+N-1}{N-1}\leq e^{N-1}\left(\cfrac{h+N-1}{N-1}\right)^{N-1}\leq C_{N}h^{N}
\]
for some constant $C_{N}$, we get
\begin{equation*}
	\mathbb{P}\left(\KL{M+S}{p}>\KL{R+S}{p}\right)\leq C_{N}^{3}h^{3N}\exp\left(-hV(\alpha,\ee)\right).
\end{equation*}
For any $\delta>0$, one can always find some constant (independent of $h$ and $V(\alpha,\ee)$) such that
\[
C_{N}^{3}h^{3N}\exp\left(-hV(\alpha,\ee)\right)\leq C \exp\left(-h(V(\alpha,\ee)-\delta)\right).
\]
This ends the proof.


\acks %
The authors would like to thank the anonymous reviewers for their valuable comments that helped them to considerably improve the preliminary version of the paper. In particular, one of the reviewers pointed out an error in the proof of Theorem~\ref{thm:optim}, which we corrected by following his/her suggestions.

\fund %
There are no funding bodies to thank relating to the creation of this article.

\competing %
There were no competing interests to declare which arose during the preparation or publication process of this article.

\bibliographystyle{acm}
\bibliography{biblio}

\end{document}